\documentclass[review,onefignum,onetabnum]{siamart190516}

\usepackage{lipsum}
\usepackage{amsfonts}
\usepackage{graphicx}
\usepackage{epstopdf}
\usepackage{algorithmic}
\ifpdf
  \DeclareGraphicsExtensions{.eps,.pdf,.png,.jpg}
\else
  \DeclareGraphicsExtensions{.eps}
  \fi

\usepackage{dsfont}
\usepackage{mathtools}
\usepackage{verbatim}
\usepackage{epsfig,epic,eepic,graphicx}
\usepackage[mathcal]{euscript}
\usepackage{import}
\usepackage{xifthen}

\usepackage{marginnote} 
\reversemarginpar

\usepackage{enumitem}

\newcommand{\ZZ}{{\mathbb Z}}

\newcommand{\RR}{{\mathbb R}}

\newcommand{\causs}{\vec{\sigma}} 

\newcommand{\emB}{\mathbf{B}}
\newcommand{\emC}{\mathbf{C}}
\newcommand{\emH}{\mathbf{H}}
\newcommand{\emD}{\mathbf{D}}

\newcommand{\emG}{\mathbf{G}}

\newcommand{\GG}{\mathrm{G}}

\newcommand{\DD}{{\mathbb{D}}}
\newcommand{\II}{{\mathbf I}}

\newcommand{\bg}{\mathbf{g}}
\newcommand{\ff}{\mathbf{f}}
\newcommand{\nn}{\mathbf{n}}
\newcommand{\ee}{\mathrm{e}}
\newcommand{\hh}{\mathbf{h}}

\newcommand{\uu}{\mathbf{u}}
\newcommand{\UU}{\mathbf{U}}
\newcommand{\vv}{\mathbf{v}}
\newcommand{\ww}{\mathbf{w}}

\renewcommand\vec[1]{\boldsymbol #1} 

\newcommand{\nRe}{{R}_{{e}}}
\newcommand{\nRm}{{R}_{{m}}}

\newcommand{\nAl}{{A}_{{l}}}

\newcommand{\per}{\mathrm{per}}
\newcommand{\hmeas}{\mathcal{H}^{d-1}}

\newcommand{\magn}{\mathrm{mag}}

\newcommand{\abs}[1]{\left\lvert #1\right\rvert}

\newcommand{\norm}[1]{\left\lVert #1\right\rVert}
\newcommand{\dualpair}[3]{\sideset{_{#2}}{_{#3}}{\mathop{\left\langle
        #1 \right\rangle}}}

\newcommand{\cv}{\xrightarrow[\hphantom{~2~}]{}}
\newcommand{\tscale}{\xrightharpoonup[]{~2~}}

\newcommand{\wcv}[1][]{%
\ifthenelse{\isempty{#1}}{\xrightharpoonup[\hphantom{~2~}]{}}{\xrightharpoonup[\hphantom{~2~}]{#1}}%
}

\newcommand{\calA}{\mathcal{A}}
\newcommand{\calB}{\mathcal{B}}
\newcommand{\calC}{\mathcal{C}}
\newcommand{\calD}{\mathcal{D}}
\newcommand{\calE}{\mathcal{E}}
\newcommand{\calH}{\mathfrak{H}}
\newcommand{\calL}{\mathcal{L}}
\newcommand{\calN}{\mathcal{N}}
\newcommand{\calM}{\mathcal{M}}

\newcommand{\calQ}{\mathcal{Q}}
\newcommand{\calP}{\mathfrak{P}}
\newcommand{\calU}{\mathfrak{U}}
\newcommand{\calV}{\mathfrak{V}}
\newcommand{\calX}{\mathfrak{X}}
\newcommand{\calY}{\mathfrak{Y}}

\DeclareMathOperator{\diam}{diam}

\DeclareMathOperator{\supp}{supp}

\DeclareMathOperator{\di}{d\!}

\DeclareMathOperator{\Div}{{div}}
\DeclareMathOperator{\Curl}{{curl}}

\DeclareMathOperator{\Bog}{{Bog}}

\def\XXint#1#2#3{{\setbox0=\hbox{$#1{#2#3}{\int}$ }
\vcenter{\hbox{$#2#3$ }}\kern-.6\wd0}}

\providecommand\given{} 
\newcommand\givensymbol[1]{%
  \nonscript\;\delimsize#1\allowbreak\nonscript\;\mathopen{}%
}
\DeclarePairedDelimiterX\Set[1]\{\}{%
  \renewcommand\given{\givensymbol{\vert}}%
  #1%
}
\allowdisplaybreaks

\newsiamremark{remark}{Remark}
\newsiamremark{hypothesis}{Hypothesis}
\crefname{hypothesis}{Hypothesis}{Hypotheses}
\newsiamthm{claim}{Claim}

\headers{Homogenization of a nonlinear coupled model}{T. Dang,
  Y. Gorb, and S. Jim\'{e}nez Bola\~{n}os}

\title{Homogenization of a non-linear strongly coupled model of
  magnetorheological fluids\thanks{Submitted to the editors DATE.
    \funding{The work of the first author was partially supported by
      the NSF through grant DMS-1350248. The work of the third author
      was partially supported by NSF grant DMS-2110036.}}}

\author{
   Thuyen Dang\thanks{Department of Mathematics, University of
    Houston, Houston, TX 77204 USA
    (\email{ttdang9@uh.edu}).} 
  \and
  Yuliya Gorb\thanks{National Science Foundation, Alexandria, VA 22314 USA
    (\email{ygorb@nsf.gov}).}
    \and
  Silvia Jim\'{e}nez Bola\~{n}os\thanks{Department of
    Mathematics, Colgate University, Hamilton, NY 13346 USA
    (\email{sjimenez@colgate.edu}).} 
}

\usepackage{amsopn}

\ifpdf
\hypersetup{
  pdftitle={Homogenization of a non-linear strongly coupled model of
  magnetorheological fluids},
  pdfauthor={T. Dang, Y. Gorb and S. Jim\'{e}nez Bola\~{n}os}
}
\fi

\nolinenumbers
\begin{document}

\maketitle

\begin{abstract}
  This paper concerns the rigorous periodic homogenization for a non-linear
  strongly coupled system, which models a suspension of magnetizable
  rigid particles in a non-conducting carrier viscous Newtonian
  fluid. The fluid drags the particles, thus alters the magnetic
  field. Vice versa, the magnetic field acts on the particles, which in turn affect the fluid via the no-slip boundary condition.  As
  the size of the particles approaches zero, it is shown that the suspension's behavior is governed by a generalized
  magnetohydrodynamic system, where the fluid is modeled by a stationary
  Navier-Stokes system, while the magnetic field is modeled by Maxwell equations. A corrector result from the theory of
  two-scale convergence allows us to obtain the limit of the product of
  several weakly convergent sequences, where the div-curl lemma, which is a typical tool in these types of problems,
  is not applicable.
\end{abstract}

\begin{keywords}
  Homogenization, two-scale convergence, Navier-Stokes equation, strong coupling, magnetic particles.
\end{keywords}

\begin{AMS}
  35B27, 74E30, 74F10, 76M50, 78M40.
\end{AMS}

\section{Introduction}
\label{sec:introduction}

This paper is a counterpart of our previous works \cite{dangHomogenizationNondiluteSuspension2021,dangGlobalGradientEstimate2021a}, where we carried out the rigorous periodic homogenization of a \emph{weakly (one-way) coupled} non-linear system modeling a non-dilute suspension of magnetizable particles in a viscous Newtonian fluid. In \cite{dangHomogenizationNondiluteSuspension2021,dangGlobalGradientEstimate2021a}, the fluid is assumed to be described by the stationary Stokes flow, and the particles are either
paramagnetic or diamagnetic. The \emph{one-way coupling} is understood as follows: the magnetic field alters the movement of the magnetizable particles, then the particles affect the fluid flow via a no-slip boundary assumption; however, the reverse effect is assumed to be negligible. For details and information about the manifestations of the one-way coupling (as well as of the full coupling), its applications, and further literature on the subject, we refer the reader to \cite{dangHomogenizationNondiluteSuspension2021,dangGlobalGradientEstimate2021a}
and references cited therein. In this paper, the \emph{full (two-way) coupling} is considered, i.e. we also take into account the reverse effect: the fluid flow pushes the particles, thus generates an induced magnetic field that acts back on the original one. The mathematical formulation of the fully coupled model of the magnetic non-dilute suspension is given in
\cref{sec:formulation} below. 

Starting with the seminal work of Einstein on the effective viscosity of a
suspension \cite{einsteinNeueBestimmungMolekuldimensionen1906}, there
have been numerous studies on this subject, ranging from formal
asymptotic analysis such as
\cite{levySuspensionSolidParticles1983,levyEinsteinlikeApproximationHomogenization1985,levyHomogenizationMechanicsNondilute1988}
to rigorous analysis,
e.g. \cite{niethammerLocalVersionEinstein2020,desvillettesMeanfieldLimitSolid2008,gorbHomogenizationRigidSuspensions2014,hoferSedimentationInertialessParticles2018,mecherbetSedimentationParticlesStokes2019,gerard-varetAnalysisViscosityDilute2020,duerinckxCorrectorEquationsFluid2021,hoferMotionSeveralSlender2021,hainesProofEinsteinEffective2012,duerinckxSedimentationRandomSuspensions2021,duerinckxEinsteinEffectiveViscosity2020,duerinckxQuantitativeHomogenizationTheory2021,berlyandFictitiousFluidApproach2009a,berlyandHomogenizedNonNewtonianViscoelastic2004}
and references cited therein. The coupling between the velocity and the
magnetic fields distinguishes our paper from the previously cited. In this paper, \emph{we propose a non-linear system to
  model the  two-way coupling in the magnetorheological fluid, and derive, and rigorously justify, the corresponding
  effective system}. 
  
  To overview the literature on this topic, we start with the phenomenological models proposed in
e.g.
\cite{neuringerFerrohydrodynamics1964,odenbachFerrofluidsMagneticallyControllable2008,eringenElectrodynamicsContinuaII2011,nochettoEquationsFerrohydrodynamicsModeling2016,grunFieldinducedTransportMagnetic2019}, 
whose well-posedness were studied in
e.g. \cite{nochettoDynamicsFerrofluidsGlobal2019,grunFieldinducedTransportMagnetic2021,garckeStrongWellposednessStability2021a,benesovaExistenceWeakSolutions2018a,schlomerkemperUniquenessSolutionsMathematical2018a}.  A coupling mechanism, similar to the one discussed in this paper, was also considered in
\cite{nikaMultiscaleModelingMagnetorheological2020}, where a different model describing fluids was used. The authors in \cite{nikaMultiscaleModelingMagnetorheological2020} though obtained the results using {\it formal
asymptotic analysis}.  Although similar models in different contexts were also studied in
\cite{vernescuMultiscaleAnalysisElectrorheological2002,francfortEnhancementElastodielectricsHomogenization2021}, to the best of our knowledge, this paper is the first one to deal with the fully-coupled model for magnetorheological fluids using the {\bf rigorous homogenization approach}. Lastly, we mention that the {\it rigorous homogenization} for the system described by  one-way fluid-particle coupling
was solved in
\cite{dangHomogenizationNondiluteSuspension2021,dangGlobalGradientEstimate2021a}
with a fairly general assumption on the smoothness of the coefficients. 

In what follows below,
after a non-linear model for the magnetorheological fluid is
established, we obtain the well-posedness and a priori estimates for its
solution by adapting the general functional analysis framework of
stationary magnetohydrodynamics, c.f.
\cite{schotzauMixedFiniteElement2004,gunzburgerExistenceUniquenessFinite1991,gerbeauMathematicalMethodsMagnetohydrodynamics2006,guermondMixedFiniteElement2003}
and references therein. Then, the two-scale convergence method, c.f.
\cite{allaireHomogenizationTwoscaleConvergence1992,nguetsengGeneralConvergenceResult1989,berlyandGettingAcquaintedHomogenization2018,cioranescuIntroductionHomogenization1999},
is utilized to obtain the effective, or {\it homogenized}, system. The main difficulty lies
in the non-linearity of the system, c.f. \eqref{eq:575-nd} and
\eqref{eq:579-nd}, and, the full coupling mechanism captured by
\eqref{eq:575-nd}, \eqref{eq:579-nd} and \eqref{eq:13} 
that make the choice of suitable
oscillating test functions in the energy method by Tartar
\cite{tartarGeneralTheoryHomogenization2009}, which is a typical tool in homogenization problems, to become extremely
tricky.  To overcome
this difficulty, we rely on the corrector result from the two-scale
convergence method, see \cref{sec:two-scale-corrector}.
The results obtained in this paper can be extended to the stochastic setting,
thanks to the work on stochastic two-scale convergence,
c.f. \cite{bourgeatStochasticTwoscaleConvergence1994,zhikovHomogenizationRandomSingular2006,heidaExtensionStochasticTwoscale2011,heidaStochasticTwoscaleConvergence2021}
and references cited therein. 

This paper is organized as follows. In
\cref{sec:formulation}, the main notations are introduced and
the formulation of the fine-scale problem is discussed. Our main
result is stated in \cref{sec:main-results}, and the conclusions
are given in \cref{sec:conclusions}.

\section{Formulation}
\label{sec:formulation}
\subsection{Notation}

Throughout this paper, the scalar-valued functions, such as the
pressure $p$, are written in usual typefaces, while vector-valued or
tensor-valued functions, such as the velocity $\uu$ and the Cauchy
stress tensor $\causs$, are written in bold.
Sequences are indexed by
  superscripts ($\phi^i$), while elements of vectors or tensors
are indexed by numeric subscripts ($x_i$). Finally, the Einstein
summation convention is used whenever applicable; $\delta_{ij}$ is the Kronecker delta, and $\epsilon_{ijk}$ is the Levi-Civita permutation symbol.
\subsection{Set up of the problem}
\label{ss:setup}
Consider $\Omega \subset \RR^d$, for $d \ge 2$, a simply connected and
bounded domain of class $C^{1,1}$, and let
$Y\coloneqq (0,1)^d$ be the unit cell in $\RR^d$. The unit cell $Y$ is
decomposed into:
$$Y=Y_s\cup Y_f \cup \Gamma,$$
where $Y_s$, representing the magnetic
inclusion, and $Y_f$, representing the fluid domain, are open sets in
$\mathbb{R}^d$, and $\Gamma$ is the closed $C^{1,1}$ interface that
separates them.
Let $i = (i_1, \ldots, i_d) \in \ZZ^d$ be a vector of indices and $\{\ee^1, \ldots, \ee^d\}$ be the canonical basis of $\RR^d$. 
For a fixed small $\varepsilon > 0,$ we define the dilated sets: 
\begin{align*}
    Y^\varepsilon_i 
    \coloneqq \varepsilon (Y + i),~~
    Y^\varepsilon_{i,s}
    \coloneqq \varepsilon (Y_s + i),~~
    Y^\varepsilon_{i,f}
    \coloneqq \varepsilon (Y_f + i),~~
    \Gamma^\varepsilon_i 
    \coloneqq \partial Y^\varepsilon_{i,s}.
\end{align*}
Typically, in homogenization theory, the positive number $\varepsilon
\ll 1$ is referred to as the {\it size of the microstructure}. The
effective or homogenized response of the given suspension
corresponds to the case $\varepsilon=0$, whose  derivation and justification is the main focus of this paper. 

We denote by $\nn_i,~\nn_{\Gamma}$ and $\nn_{\partial \Omega}$ the unit normal vectors to $\Gamma^{\varepsilon}_{i}$ pointing outward $Y^\varepsilon_{i,s}$, on $\Gamma$ pointing outward $Y_{s}$ and on $\partial \Omega$ pointing outward, respectively; and also, we denote by $\di \hmeas$ the $(d-1)$-dimensional Hausdorff measure.
In addition, we define the sets:
\begin{align*}
    I^{\varepsilon} 
    \coloneqq \{ 
    i \in \ZZ^d \colon Y^\varepsilon_i \subset \Omega
    \},~~
    \Omega_s^{\varepsilon} 
    \coloneqq \bigcup_{i\in I^\varepsilon}
Y_{i,s}^{\varepsilon},~~
    \Omega_f^{\varepsilon} 
    \coloneqq \Omega \setminus \Omega_s^{\varepsilon},~~
    \Gamma^\varepsilon 
    \coloneqq \bigcup_{i \in I^\varepsilon} \Gamma^\varepsilon_i.
\end{align*}
see \cref{fig:1}.

\begin{figure}
\centering
\includegraphics[scale=.35]{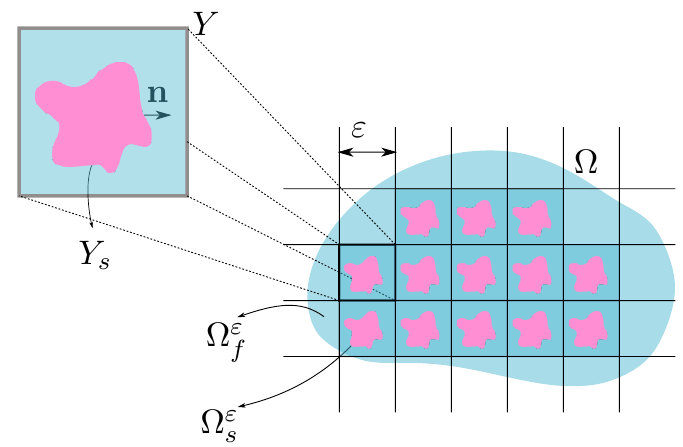}
\caption{Reference cell $Y$ and domain $\Omega$.}
\label{fig:1}
\end{figure}


\subsection{The model}
\label{sec:simplified-model}
Denote by $\rho_f,\rho_s, \nu_{es}, \mu$ and $\bg$ the (mass)
density of fluid, the density of inclusions, the electric conductivity of
inclusions, the magnetic permeability and the external force field, respectively.  The
unknowns include the fluid velocity $\uu^{\varepsilon}$, the fluid pressure
$p^{\varepsilon}$ and the magnetic field $\emB^{\varepsilon}$ (which in
turn, determines the magnetising field $\emH^{\varepsilon}$).
For simplicity, we assume that the magnetic permeability is piecewise-constant and given by
\[
\mu(x)=
\begin{cases}
    \mu_f, &\text{ if } x \in
    \Omega^\varepsilon_f,\\
    \mu_s, &\text{ if } x \in
    \Omega^\varepsilon_s,
  \end{cases}
\]
where $\mu_f,\,\mu_s >0$.

We
consider the following non-linear system modeling a suspension of
rigid inclusions in a non-conducting carrier fluid:
\begin{subequations}
  \label{eq:ferro-eqn-sim}  
\begin{align}
\label{eq:575}
\rho_f \left[ \frac{\partial \uu^{\varepsilon}}{\partial t} + (\uu^{\varepsilon} \cdot \nabla) \uu^{\varepsilon}
  \right] - \Div  \causs^{\varepsilon}
  &= \rho_f \bg
  &&\text{ in }\Omega^{\varepsilon}_f,\\
  \label{eq:576}
  \Div \uu^{\varepsilon}
  &= 0
  &&\text{ in } \Omega^{\varepsilon}_f\\
  \label{eq:577}
  \DD(\uu^{\varepsilon})
  &= 0
  &&\text{ in }\Omega^{\varepsilon}_s,\\
  \label{eq:590}
  \Curl \emH^{\varepsilon}
  &= 0
  &&\text{ in }\Omega^{\varepsilon}_f,\\
  \label{eq:578}
  \frac{\partial \emB^{\varepsilon}}{\partial t} + \frac{1}{\nu_{es}} \Curl \Curl
  \emH^{\varepsilon}
  &= \Curl \left( \uu^{\varepsilon} \times \emB^{\varepsilon} \right)
  &&\text{ in } \Omega^{\varepsilon}_s,\\
  \label{eq:579}
  \Div \emB^{\varepsilon}
  &= 0
  &&\text{ in }\Omega,\\
  \label{eq:574}
  \emB^{\varepsilon}
  &= \mu \emH^{\varepsilon}
  &&\text{ in }\Omega.
\end{align}
\end{subequations}
Suppose that the induced electric field is negligible, then the
Lorentz force can be written as:
\begin{align}
\label{eq:572}
  \ff^{\magn} = \Curl \emH^{\varepsilon} \times \emB^{\varepsilon}=
  \begin{cases}
    0, &\text{ in }\Omega^{\varepsilon}_f,\\
    \Curl \emH^{\varepsilon} \times \emB^{\varepsilon}, &\text{ in }\Omega^{\varepsilon}_s.
  \end{cases}
\end{align}
Thus by the Law of Inertia \cite[Axiom 5.2, page
171]{gonzalezFirstCourseContinuum2008a}, we obtain the balance
equations of force and torque: 
\begin{subequations}
  \label{eq:ferro-balance}  
  \begin{align}
\label{eq:581}
  \int_{Y_{i,s}} \rho_s \dot{\uu^{\varepsilon}} \di x
    &= \int_{\Gamma_i}
      \causs^{\varepsilon}(\uu^{\varepsilon},p^{\varepsilon})\nn\di \Gamma + \int_{Y_{i,s}}
    \Curl \emH^{\varepsilon} \times \emB^{\varepsilon} \di x + \int_{Y_{i,s}} \rho_s \bg \di x,\\
  \int_{Y_{i,s}} \rho_s \left( x - {\GG_i} \right)\times \dot{\uu^{\varepsilon}} \di x
    &= \int_{\Gamma_i} \left( x - {\GG_i} \right)\times
   \causs^{\varepsilon}(\uu^{\varepsilon},p^{\varepsilon})\nn \di \Gamma \nonumber\\
  &\qquad{} +{} \int_{Y_{i,s}}
    \left( x - {\GG_i} \right)\times (\Curl \emH^{\varepsilon} \times \emB^{\varepsilon}) \di
    x \nonumber\\
  &\qquad {}+{}\int_{Y_{i,s}} \rho_s \left( x - {\GG_i} \right)
    \times \bg \di x,
\end{align}
\end{subequations}
where $\dot{\uu^\varepsilon} \coloneqq \frac{\partial \uu^{\varepsilon}}{\partial t} + \left(\uu^\varepsilon \cdot \nabla\right)\uu^\varepsilon$ is the convective derivative, and $\GG_i$ is the center of mass of the particle $Y_{i,s}$.
The (outer) boundary conditions on the external boundary $\partial \Omega$ are
\begin{align}
\label{eq:583}
\uu^{\varepsilon} = 0, \quad \Curl \emH^{\varepsilon} \times \nn = 0, \quad \emB^{\varepsilon} \cdot \nn = q,
\end{align}
where 
$\displaystyle  \causs^{\varepsilon} (\uu^{\varepsilon}, p^{\varepsilon}) \coloneqq 2
  \eta \DD(\uu^{\varepsilon}) - p^{\varepsilon} I$, $\displaystyle   
   \DD(\uu^{\varepsilon})
   \coloneqq \frac{\nabla \uu^\varepsilon + \left(\nabla \uu^\varepsilon \right)^\top}{2}$, 
and $q \in H^{1/2}(\partial \Omega)$ satisfying the compatibility condition $\int_{\partial \Omega} q \di \hmeas = 0.$
Since $\mu$ is piecewise-constant, from now on, we write: 
\begin{align*}
\emH^{\varepsilon} =
  \begin{cases}
\displaystyle     \frac{\emB^{\varepsilon}}{\mu_f}, &\text{ if } x \in
    \Omega^\varepsilon_f,\\[12pt]
\displaystyle      \frac{\emB^{\varepsilon}}{\mu_s}, &\text{ if } x \in
    \Omega^\varepsilon_s.
  \end{cases}
\end{align*}

\subsection{Dimensional analysis}
\label{sec:dimensional-analysis}

Let $L, U, B$, and $\mu_s$ be the characteristic scales corresponding to
length, velocity, magnetic field and magnetic permeability,
respectively.  The characteristic time $T$ and body density force $F$
are defined by $T = \frac{L}{U}$ and $F = \frac{U^2}{L}.$

Let
$x^{*} \coloneqq \frac{x}{L}, \uu^{\varepsilon *} \coloneqq
\frac{\uu}{U}, p^{\varepsilon *}
\coloneqq \frac{p L}{\eta U}, \bg^{*} \coloneqq \frac{\bg L}{U^2}$ and
$\mu^{*} \coloneqq \frac{\mu}{\mu_s}$. The dimensionless quantities
that appear are the \emph{hydrodynamic Reynolds number}
$\nRe= \frac{\rho_f UL}{\eta}$, the \emph{magnetic Reynolds number}
$\nRm = \mu_s \nu_{es} UL$, the \emph{Alfven number}
$\nAl = \frac{B^2L}{\eta \mu_s U}$, and the \emph{density ratio} which, for simplicity, is assumed to satisfy $
\frac{\rho_s}{\rho_f} =1$.  Also, in the sequel, we drop the star to lighten the notation. The
dimensionless versions of \eqref{eq:ferro-eqn-sim},
\eqref{eq:ferro-balance} and \eqref{eq:583} are:
\begin{subequations}
  \label{eq:ferro-eqn-sim-nd}  
\begin{align}
\label{eq:575-nd}
 \nRe \left[ \frac{\partial \uu^{\varepsilon}}{\partial t} +
  (\uu^{\varepsilon} \cdot \nabla) \uu^{\varepsilon}  \right]
   - \Div \causs^{\varepsilon} (\uu^{\varepsilon},p^{\varepsilon})
  &= \nRe \bg
  &&\text{ in }\Omega^{\varepsilon}_f,\\
  \label{eq:576-nd}
  \Div \uu^{\varepsilon}
  &= 0
  &&\text{ in } \Omega^{\varepsilon}_f\\
  \label{eq:577-nd}
  \DD(\uu^{\varepsilon})
  &= 0
  &&\text{ in }\Omega^{\varepsilon}_s,\\
  \label{eq:590-nd}
  \Curl \emB^{\varepsilon}
  &= 0
  &&\text{ in }\Omega^{\varepsilon}_f,\\
  \label{eq:578-nd}
  \frac{\partial \emB^{\varepsilon}}{\partial t} + \frac{1}{\nRm} \Curl \Curl
  \emB^{\varepsilon}
  &= \Curl \left( \uu^{\varepsilon} \times \emB^{\varepsilon} \right)
  &&\text{ in } \Omega^{\varepsilon}_s,\\
  \label{eq:579-nd}
  \Div \emB^{\varepsilon}
  &= 0
  &&\text{ in }\Omega.
\end{align}
\end{subequations}
equipped with the balance equations:
\begin{subequations}
  \label{eq:ferro-balance-nd}  
  \begin{align}
\label{eq:581-nd}
  \nRe \int_{Y^{\varepsilon}_{i,s}}  \dot{\uu^{\varepsilon}} \di x
    &= \int_{\Gamma^{\varepsilon}_i} 
      \causs^{\varepsilon}(\uu^{\varepsilon},p^{\varepsilon})\nn\di
      \hmeas \nonumber\\
  &\qquad{}+{} \nAl \int_{Y^{\varepsilon}_{i,s}}
    \Curl \emB^{\varepsilon} \times \emB^{\varepsilon} \di x +  \nRe\int_{Y^{\varepsilon}_{i,s}} \bg \di x,\\
 \nRe \int_{Y^{\varepsilon}_{i,s}} \left( x - {\GG^{\varepsilon}_i} \right)\times \dot{\uu^{\varepsilon}} \di x
    &= \int_{\Gamma^{\varepsilon}_i} \left( x - {\GG^{\varepsilon}_i} \right)\times
    \causs^{\varepsilon}(\uu^{\varepsilon},p^{\varepsilon}) \nn \di \hmeas  +  \nonumber\\
  &\qquad{}+{} \nAl \int_{Y^{\varepsilon}_{i,s}}
    \left( x - {\GG^{\varepsilon}_i} \right)\times (\Curl \emB^{\varepsilon} \times \emB^{\varepsilon}) \di
    x\nonumber\\
  &\qquad{} +{} \nRe\int_{Y^{\varepsilon}_{i,s}} \left( x - {\GG^{\varepsilon}_i} \right) \times \bg \di x,
\end{align}
\end{subequations}
 and the boundary conditions: 
\begin{align}
\label{eq:583-nd}
\uu^{\varepsilon} = 0, \quad \Curl \emB^{\varepsilon} \times \nn = 0, \quad \emB^{\varepsilon} \cdot \nn = q,
\end{align}
where now 
$\displaystyle  
  \causs^{\varepsilon }(\uu^{\varepsilon},p^{\varepsilon})
  \coloneqq 2\DD (\uu^{\varepsilon}) -
     p^{\varepsilon } \II 
$.

Hereafter, we consider the stationary flow, i.e.,
the time derivative is ignored:
\begin{subequations}
  \label{eq:5}  
\begin{align}
\label{eq:6}
  \nRe (\uu^{\varepsilon} \cdot \nabla) \uu^{\varepsilon} -\Div
  \causs^{\varepsilon} (\uu^{\varepsilon},p^{\varepsilon})
  &=\nRe\bg
  && \text{ in }\Omega^{\varepsilon}_f,\\
  \label{eq:7}
  \Div \uu^{\varepsilon}
  &= 0
  && \text{ in } \Omega^{\varepsilon}_f,\\
  \label{eq:8}
  \DD (\uu^{\varepsilon})
  &= 0
  && \text{ in } \Omega_s^{\varepsilon},\\
  \label{eq:9}
  \Curl \emB^{\varepsilon}
  &= 0
  && \text{ in } \Omega_f^{\varepsilon},\\
  \label{eq:10}
  \frac{1}{\nRm} \Curl \Curl \emB^{\varepsilon}
  - \Curl \left( \uu^{\varepsilon}\times \emB^{\varepsilon} \right)
  &= \hh
  && \text{ in } \Omega^{\varepsilon}_s,\\
  \label{eq:11}
  \Div \emB^{\varepsilon}
  &= 0
  && \text{ in } \Omega.
\end{align} 
\end{subequations}
equipped with the balance equations:
\begin{subequations}
  \label{eq:13}  
\begin{align}
\label{eq:14}
 \nRe\int_{Y^{\varepsilon}_{i,s}}(\uu^{\varepsilon} \cdot \nabla) \uu^{\varepsilon}
  &=\int_{\Gamma^{\varepsilon}_i}  \causs^{\varepsilon}\nn \di \hmeas \nonumber\\
  &\qquad{}+{} \nAl
  \int_{Y^{\varepsilon}_{i,s}} \Curl \emB^{\varepsilon}\times
  \emB^{\varepsilon} \di x + \nRe \int_{Y^{\varepsilon}_{i,s}} \bg
    \di x,\\
  \nRe \int_{Y^{\varepsilon}_{i,s}}(x-\GG^{\varepsilon}_i) \times (\uu^{\varepsilon} \cdot \nabla) \uu^{\varepsilon}
  &= \int_{\Gamma^{\varepsilon}_i} \left( x - \GG^{\varepsilon}_i
    \right) \times  \causs^{\varepsilon} \nn \di \hmeas
   \nonumber\\
  &\qquad{}+{} \nAl \int_{Y^{\varepsilon}_{i,s}} \left( x - \GG_i^{\varepsilon}
    \right) \times \left( \Curl \emB^{\varepsilon} \times
    \emB^{\varepsilon} \right) \di x \nonumber\\
  &\qquad{}+{} \nRe
    \int_{Y^{\varepsilon}_{i,s}} (x - \GG^{\varepsilon}_i) \times \bg
    \di x,
\end{align}
\end{subequations}
the boundary conditions
\begin{align}
\label{eq:583s}
\uu^{\varepsilon} = 0, \quad \Curl \emB^{\varepsilon} \times \nn = 0, \quad \emB^{\varepsilon} \cdot \nn = 0
\end{align}
and the compatibility condition 
\begin{align}
\label{eq:596}
\int_{\Omega_s}\hh \cdot \nabla \psi \di x  = 0 \text{ for all }
  \nabla \psi
  \in H^1_n(\Omega,\RR^d).
\end{align}
We note that the function $\hh \in L^2(\Omega, \RR^d)$ appears in \eqref{eq:10} due to a
{\it lifting} of the non-homogeneous magnetic condition \eqref{eq:583} to the homogeneous condition \eqref{eq:583s}, i.e. substracting $\emB^\varepsilon$ by a suitable function, see \cite{gunzburgerExistenceUniquenessFinite1991}, and \cite[Section 3.8]{gerbeauMathematicalMethodsMagnetohydrodynamics2006}. Here, $H^1_n(\Omega,\RR^d)$ is
the set of weakly differentiable functions from $\Omega$ to $\RR^d$
with vanishing normal trace, see \cref{sec:abstr-fram} below.

\subsection{Useful results from functional analysis}
\label{sec:usef-exist-conv}

In this section, we collect some background results from functional analysis used
in the sequel. We separate the functional spaces and theorems of the two-scale convergence method from the ones of saddle point problems
to make it easier to keep track.

\subsubsection{Abstract framework for our non-linear problem}
\label{sec:abstr-fram}
The results for linear saddle point problems date back to the seminal
works by I. Babu\v{s}ka and F. Brezzi,
c.f. \cite{babuskaFiniteElementMethod1972,boffiMixedFiniteElement2013}. They
are then adapted to the non-linear cases such as the Navier-Stokes
equations and magnetorhydrodynamic equations,
c.f. \cite{giraultFiniteElementMethods2012,schotzauMixedFiniteElement2004,gunzburgerExistenceUniquenessFinite1991,gerbeauMathematicalMethodsMagnetohydrodynamics2006,guermondMixedFiniteElement2003}.
We summarize here the results used in our paper and refer the readers
to the works cited above for their proofs.

Let $X$ and $P$ be two real Hilbert spaces, and $f \in X$.
Let $a(\,\cdot\,;\cdot,\cdot)\colon X \times X \times X \to \RR$ be a
non-linear form such that for any $w \in X$, $a(w;\cdot,\cdot)$ is a
bilinear continuous form on $X \times X$. Let $b\colon X \times P \to \RR$ be
a continuous bilinear form.
Consider the following
non-linear problem:

\emph{Find $(u,p) \in X \times P$, such that for
  all $(v,q) \in X \times P$,}
\begin{subequations}
  \label{eq:33}
\begin{align}
  \label{eq:31}
  a(u;u,v) + b(v,p) &= \left\langle f,v \right\rangle,\\
  \label{eq:32}
    b(u,q) &= 0,
\end{align}
\end{subequations}
where $\langle \cdot,\cdot \rangle$ is the dual pairing.
The unknown $p$ can be regarded as the Lagrange multiplier associated
with the constraint \eqref{eq:32}. The idea is to embed the constraint
\eqref{eq:32} into $X$ by introducing the space
\begin{align*}
M = \left\{ u \in X \colon b(u,q) = 0 \text{ for all } q \in P \right\},
\end{align*}
and consider a simpler problem that reads: \emph{Find $u \in M$ such that for
all $v \in M,$}
\begin{align}
\label{eq:598}
    a(u;u,v) 
    = \left\langle f,v \right\rangle.
\end{align}
The continuity of $b$ implies that $M$ is a closed linear subspace of
$X$, and, thus, $M$ is also a Hilbert space.

\begin{theorem}[Existence and uniqueness of solution of
  \eqref{eq:598}
]
\label{sec:stat-ferr-equat-3}
If the following conditions hold:
\begin{enumerate}
\item[(i)] there exists $\alpha > 0 $ such that for all $v \in M$, 
\begin{align}
\label{eq:599}
a(v;v,v) \ge \alpha \norm{v}^2_X;
\end{align}
\item[(ii)] the space $M$ is separable and such that for any sequence $v_n$ that
  weakly converges to $v$ in $M$, $a(v_n;v_n,w)$ converges to
  $a(v;v,w)$, for  all $w\in M$;
\end{enumerate}
then there exists at least one solution of problem
\eqref{eq:598}: $u\in M$. If in addition, we assume that:
\begin{enumerate}
\item[(iii)] the elliptic property (i) holds uniformly with respect to
  the first variable, i.e. there exists $\alpha > 0$ such that for all
  $v,w \in M$, 
\begin{align}
\label{eq:600}
a(w;v,v) \ge \alpha \norm{v}^2_X,
\end{align}
\item[(iv)] there exists a constant $\gamma > 0$ such that, for all
  $u_1, u_2, v, w \in M$,
\begin{align}
\label{eq:601}
  \abs{a(u_2;v,w) - a(u_1;v,w)}
  \le \gamma \norm{u_2 - u_1}_X \norm{v}_X \norm{w}_X,
\end{align}
\end{enumerate}
then problem \eqref{eq:598} has a unique solution $u \in M$, provided
that 
\begin{align}
\label{eq:602}
\frac{\gamma \norm{w}_M}{\alpha^2}<1,
\end{align}
where $w\in M$ is such that $\langle f|_M, v\rangle = ( w, v )_X, ~ \forall v \in M$, with $f|_M$ being the restriction of $f$ on $M$.
\end{theorem}

 \cref{sec:stat-ferr-equat-3} allows us to establish the
 existence and uniqueness of the solution $u$ of \eqref{eq:598}. To recover the unknown $p$ that solves \eqref{eq:33}, we need to introduce the following definition.
\begin{definition}
\label{sec:abstr-fram-our}
The following is called the inf-sup condition or the
Babu\v{s}ka-Brezzi condition or the Ladyzhenskaya-Babu\v{s}ka-Brezzi
condition: 
\begin{align}
\label{eq:35}
  \exists \, \beta > 0 \quad \text{ such that } \quad
  \inf_{q \in P\setminus \left\{ 0 \right\}}
  \sup_{v \in X\setminus \left\{ 0 \right\}} \frac{b(v,q)}{\norm{v}_X
  \norm{q}_P}\ge \beta.
\end{align}
\end{definition}

If the bilinear form $b$ in \eqref{eq:33} satisfies the inf-sup
condition \eqref{eq:35} then, by the Riesz Representation Theorem and
the Closed Range Theorem \cite{brezisFunctionalAnalysisSobolev2011}, the
existence and uniqueness of the solution $u$ of \eqref{eq:598} implies the
existence and uniqueness of the solution $(u,p)$ of \eqref{eq:33}. 

The inf-sup condition can be verified by
\begin{proposition}
\label{sec:check-infsup}
Let $B\colon X \to P$ be the continuous linear operator associated to
the continuous bilinear form $b$ by $( Bv,q )_P =
b(v,q)$ for all $(v,q) \in X \times P$ (here we use the Riesz
Representation Theorem). Then the following statements are equivalent: 
\begin{enumerate}
\item[(i)] The inf-sup condition \eqref{eq:35} holds.
\item[(ii)] $B^\top \colon P \to X$ is injective and $B^\top$ has a closed
  range. Here $B^\top$ is the transpose of $B$, i.e. $( v, B^\top q )_X = ( Bv, q)_P$ for all $(v,q) \in X \times P.$
\item[(iii)] $B \colon X \to P$ is surjective. 
\end{enumerate}
\end{proposition}

\subsubsection{The two-scale convergence method}
\label{sec:two-scale-conv}
Two-scale convergence was invented by G. Nguetseng and
further developed by G. Allaire.  We collect here the important notions
and results relevant to this paper, whose proofs can be found in \cite{bourgeatStochasticTwoscaleConvergence1994,zhikovHomogenizationRandomSingular2006,heidaExtensionStochasticTwoscale2011,heidaStochasticTwoscaleConvergence2021}. The following spaces are used in the paper below.
\begin{itemize}[wide]
    \item $C_{\per}(Y)$ -- the subspace of $C(\RR^d)$ of $Y$-periodic functions;
    \item $C^{\infty}_{\per}(Y)$ -- the subspace of $C^{\infty}(\RR^d)$ of $Y$-periodic functions;
    \item $H^1_{\per}(Y)$ -- the closure of $C^{\infty}_{\per}(Y)$ in the $H^1$-norm;
    \item
    
    $\mathcal{D}(\Omega, X)$ -- where $X$ is a Banach space -- the space  infinitely differentiable functions from $\Omega$ to $X$, whose  support is a compact set of $\mathbb{R}^d$ contained in $\Omega$.

    \item $L^p(\Omega, X)$ -- where $X$ is a Banach space and $1 \le p \le \infty$ -- the space of measurable functions $w \colon x \in \Omega \mapsto w(x) \in X$ such that
    $
    \norm{w}_{L^p(\Omega, X)}
    \coloneqq \left(\int_{\Omega} \norm{w(x)}^p_{X} \di x\right)^\frac{1}{p} < \infty.
    $

    \item $L^p_{\per}\left(Y, C(\bar{\Omega})\right)$ -- the space of measurable functions $w \colon y \in Y \mapsto w(\cdot,y) \in C(\bar{\Omega})$, such that 
    $w$ is periodic with respect to $y$ and
    $
    \int_{Y} \left(\sup_{x \in \bar{\Omega}} \abs{w(x,y)}\right)^p \di y 
    < \infty.
    $
    
\end{itemize}

\begin{definition}[$L^p-$admissible test function]
\label{sec:two-scale-conv-1}
Let $1 \le p < + \infty$. A function $\psi \in
L^p(\Omega \times Y)$,  $Y$-periodic in the second component, is called an $L^p-$admissible test function if for
all $\varepsilon > 0$,
$\psi \left( \cdot, \frac{\cdot}{\varepsilon} \right)$ is measurable and
\begin{align}
\label{eq:34}
\lim_{\varepsilon \to 0} \int_{\Omega} \abs{\psi \left( x,
  \frac{x}{\varepsilon} \right)}^p \di x = \frac{1}{\abs{Y}}
  \int_{\Omega} \int_Y \abs{\psi (x,y)}^p \di y \di x.
\end{align}
\end{definition}

It is known that functions belonging to the spaces $\mathcal{D} \left(\Omega,
  C_\per^\infty (Y)\right)$, $C \left( \bar{\Omega}, C_{\per}(Y)
\right)$, $L^p_{\per}\left( Y,  C(\bar{\Omega})\right)$ or $L^p\left(\Omega, C_{\per}(Y)\right)$ are admissible \cite{allaireHomogenizationTwoscaleConvergence1992}, but the precise
characterization of those admissible test functions is still an
open question.
\begin{definition}
A sequence $\{ v^\varepsilon \}_{\varepsilon>0}$ in $L^2(\Omega)$ is said to \emph{two-scale converge} to $v = v(x,y)$, with $v \in L^2 (\Omega \times Y)$, and we write $v^\varepsilon \tscale v$, if and only if:
\begin{align}
\label{eq:2sc}
    \lim_{\varepsilon \to 0} \int_\Omega v^\varepsilon(x) \psi \left( x, \frac{x}{\varepsilon}\right) \di x 
    = \frac{1}{\abs{Y}} \int_\Omega \int_Y v(x,y) \psi(x,y) \di y \di x,
\end{align}
for any test function $\psi = \psi (x, y)$ with $\psi \in
\mathcal{D} \left(\Omega, C_\per^\infty (Y)\right)$.
\end{definition}
In \eqref{eq:2sc}, we can choose $\psi$ be any ($L^2-$)admissible test
function. Any bounded sequence $v^{\varepsilon}\in L^2(\Omega)$ has
a subsequence that two-scale converges to a limit $v^0 \in L^2(\Omega
\times Y)$. Moreover, from \cite[Theorem 1.8, Remark 1.10 and
Corollary 5.4]{allaireHomogenizationTwoscaleConvergence1992}, we have 
\begin{theorem}[Corrector result]
\label{sec:two-scale-corrector}
Let $u^{\varepsilon}$ be a sequence of functions in $L^2(\Omega)$ that
two-scale converges to a limit $u^0 (x,y) \in L^2(\Omega\times
Y)$. Assume that 
\begin{align}
\label{eq:12}
\lim_{\varepsilon \to 0} \norm{u^{\varepsilon}}_{L^2(\Omega)} =
  \norm{ 
  u^0
  }_{L^2(\Omega\times Y)}.
\end{align}
Then for any sequence $v^{\varepsilon}$ in $L^2(\Omega)$ that
two-scale converges to $v^0 \in L^2(\Omega \times Y),$ one has 
\begin{align}
\label{eq:36}
u^{\varepsilon} v^{\varepsilon} \wcv \frac{1}{\abs{Y}} \int_Y u^0(x,y)
  v^0(x,y) \di x \di y \text{ in } \calD '(\Omega).
\end{align}
Furthermore, if $u^0(x,y)$ 
belongs to $L^2 \left( \Omega, C_{\per}(Y) \right)$ or $L^2_{\per} \left( Y, C(\Omega) \right)$,
then 
\begin{align}
\label{eq:37}
\lim_{\varepsilon \to 0}  \norm{u^{\varepsilon}(x) - u^0 \left( x,
  \frac{x}{\varepsilon} \right)}_{L^2(\Omega)} = 0.
\end{align}
\end{theorem}
In fact, the smoothness assumption on $u^0$ in \eqref{eq:37} is needed
only for $u^0 \left( x, \frac{x}{\varepsilon}
\right)$ is to be measurable and to belong to $L^2(\Omega)$. 
Finally, we recall that if $\psi \in L^2(\Omega \times Y)$ is a Carath\'{e}odory
function then $\psi \left( \cdot, \frac{\cdot}{\varepsilon} \right)$
is measurable. This fact is used later on to prove that $\mathds{1}_{\Omega \times Y_s}$ is an admissible test function.
\section{Main results}
\label{sec:main-result}

We now define the admissible spaces on the $C^{1,1}-$domain $\Omega$ for the fluid velocity $\uu^{\varepsilon}$, the magnetic field
$\emB^{\varepsilon}$ and the fluid pressure $p^{\varepsilon}$.
Let
\begin{align*}
     H^1_n(\Omega,\RR^d)
  &\coloneqq \left\{ \emC \in H^1(\Omega,\RR^d)\colon \emC \cdot
    \nn_{\partial \Omega} = 0 \right\},\\
  L_0^2(\Omega)
  &\coloneqq \left\{ q \in L^2(\Omega)\colon \int_{\Omega} q \di x =
    0\right\},\\
 \calH
  &\coloneqq \left\{
    \emC \in L^2(\Omega,\RR^d) \colon \Curl \emC \in
    L^2(\Omega,\RR^d), \Div \emC \in L^2(\Omega), \emC \cdot
    \nn_{\partial \Omega} = 0
    \right\},\\
  \calV^{\varepsilon}
  &\coloneqq \left\{ \vv \in H_0^1(\Omega,\RR^d)\colon \DD(\vv) = 0 \text{
  in } \Omega_s^{\varepsilon} \right\},\\
  \calP^{\varepsilon}
  &\coloneqq \Div(\calV^{\varepsilon})=\left\{ q \in L_0^2(\Omega)\colon \exists \vv \in
    \calV^{\varepsilon} \text{ such that } q = \Div \vv \right\},\\
  \calU^{\varepsilon}
    &\coloneqq \left\{ \vv \in
      H_0^1(\Omega,\RR^d)\colon \DD(\vv) = 0 \text{
        in } \Omega_s^{\varepsilon},~ \Div \vv = 0 \text{ in
      }\Omega^{\varepsilon}_f \right\},\\
        \calX^{\varepsilon}
  &\coloneqq \calV^{\varepsilon} \times \left\{ \emC \in
     H^1_n(\Omega,\RR^d)\colon \Curl \emC = 0 \text{ in }\Omega^{\varepsilon}_f \right\},\\
    \calY^{\varepsilon}
    &\coloneqq
    \calU^{\varepsilon}\times \left\{ \emC \in
      H^1_n(\Omega,\RR^d)\colon \Curl \emC = 0 \text{ in
      }\Omega^{\varepsilon}_f \right\}.
\end{align*}
These spaces are equipped with natural Sobolev norms. 
Moreover, given normed spaces $A$ and $B$, the norm of its product space $A \times B$ is defined by $\norm{(a,b)}^2_{A \times B} \coloneqq  \norm{a}_A^2 + \norm{b}_B^2 $ for $a \in A, b\in B.$

As we will see later, to utilize the framework presented in
\cref{sec:abstr-fram}, we choose $X = \calX^{\varepsilon}$, $M =
\calY^{\varepsilon}$, and $P = \calP^{\varepsilon}$.

In addition, let $\kappa_{GR}^{-1}$, be the norm of the embedding
$\calH \to H^1$, $\kappa_S$ the norm of the Sobolev embedding $H^1 \to L^4$, and $\kappa_K^{-1}$ the constant
in Korn's inequality, respectively.  Then the main result of this paper is summarized in
the following theorem.
\begin{theorem}
  \label{sec:main-results}
  Suppose the data $\bg$ and $\hh$ are small enough such that 
\begin{align}
\label{eq:79}
  \nRe \norm{\bg}_{L^2} + \nAl \norm{\hh}_{L^2} \le \frac{\left( \min
  \left\{ \frac{\nAl}{\nRm} \kappa_{GR},\kappa_K \right\}
  \right)^2}{\kappa_S \max \left\{ 1,2\nAl \right\}}.
\end{align}
Then, for $\varepsilon > 0$, the system \eqref{eq:5} has a unique
solution $\uu^{\varepsilon} \in H_0^1(\Omega,\RR^d),$
$p^{\varepsilon} \in L_0^2(\Omega)$,
$\emB^{\varepsilon} \in H_n^1(\Omega,\RR^d)$. Moreover, there exist a
constant, symmetric and elliptic fourth rank tensor $\calN$, and two
constant, symmetric and elliptic matrices $\calM, \calE$ such that
\begin{align}
\label{eq:80}
  \uu^{\varepsilon}
  \wcv \uu^0 \text{ in } H^1(\Omega,\RR^d),
  \quad
  \emB^{\varepsilon}
  \wcv \emB^0 \text{ in } H^1(\Omega,\RR^d),
  \quad p^{\varepsilon}
  \wcv \Pi \text{ in } L^2_0(\Omega)
\end{align}
where $\uu^0 \in H_0^1(\Omega,\RR^d)$, $\Pi \in L_0^2(\Omega)$, and  $\emB^0 \in H_n^1(\Omega,\RR^d)$ satisfy the following effective system of equations all defined on the domain $\Omega$,
\begin{align}
\label{eq:78}
\begin{split}
  \Div \uu^0 = \Div \emB^0
  &= 0,\\
  \left( \uu^0 \cdot \nabla \right)\uu^0 - \Div \left( 2
     \calN_{ijmn} \left[ \DD (\uu^0) \right]_{ij}\ee^m \otimes \ee^n -
     \Pi\, \II \right)
   &=\nRe \bg + \nAl \frac{\abs{Y_s}}{\abs{Y}}
   \Curl \emB^0 \times \emB^0,\\
   \frac{1}{\nRm} \Curl \left( \calM_{jn}\epsilon_{ijk} \frac{\partial
       B^0_i}{\partial x_k} \ee^n \right) - \Curl \left( \calE_{kn}
     \epsilon_{ijk} u^0_i B^0_j \ee^n \right)
   &= \frac{\abs{Y_s}}{\abs{Y}} \hh.
 \end{split}
\end{align}

\end{theorem}

The road map of the proof of Theorem \ref{sec:main-results} goes as follows: 
\begin{itemize}
\item First, we present the variational formulation for problem
  \eqref{eq:5}--\eqref{eq:596} and prove their equivalence in
  \cref{sec:vari-form}.
\item Second, the existence and a priori estimates for the fine-scale velocity
  $\uu^{\varepsilon}$ and the magnetic field $\emB^{\varepsilon}$ are
  established in \cref{sec:exist-priori-estim}, thanks to
  \cref{sec:stat-ferr-equat-3}. The first two steps are adapted from
  the classical theory of magnetohydrodynamics, c.f. \cite{giraultFiniteElementMethods2012,schotzauMixedFiniteElement2004,gunzburgerExistenceUniquenessFinite1991,gerbeauMathematicalMethodsMagnetohydrodynamics2006,guermondMixedFiniteElement2003}. In particular, the presentation of those two steps is inspired by \cite{gerbeauMathematicalMethodsMagnetohydrodynamics2006,schotzauMixedFiniteElement2004,gunzburgerExistenceUniquenessFinite1991}.
\item Third, in \cref {sec:exist-priori-estim-2}, the existence and a priori estimate for the fine-scale pressure $p^{\varepsilon}$ are
  recovered by an inf-sup condition. A construction based on the  Bogovski\u{\i} map allows us to control the norm of the pressure
  $p^{\varepsilon}$ uniformly with respect to $\varepsilon$, c.f. \cite{acostaSolutionsDivergenceOperator2006,acostaDivergenceOperatorRelated2017,duerinckxQuantitativeHomogenizationTheory2021,duerinckxCorrectorEquationsFluid2021,duerinckxEffectiveViscosityRandom2021a,bellaInverseDivergenceHomogenization2021,hoferMotionSeveralSlender2021}.
\item Next, the two-scale homogenized problem is derived in
  \cref{sec:two-scale-homog}. Here, a corrector result of two-scale
  convergence \cite{allaireHomogenizationTwoscaleConvergence1992} is
  crucial for passing to the limit of several integrals over a changing
  domain.
\item Finally, the local and homogenized problems are recovered in
  \cref{sec:local-problem} and
  \cref{sec:homogenized-problem}. Explicit formulas for the \emph{effective viscosity} $\calN$, the \emph{effective magnetic reluctivity} $\calM$ and the \emph{effective electric conductivity} $\calE$ are provided in \eqref{eq:77}.
\end{itemize}

\subsection{Variational formulation}
\label{sec:vari-form}

We define bilinear, trilinear, and linear forms $\calA^{\varepsilon}(\cdot,\cdot): \calX^{\varepsilon} \times
  \calX^{\varepsilon}\to \mathbb{R}$, $\calB^{\varepsilon}(\cdot,\cdot):\calX^{\varepsilon} \times
  \calP^{\varepsilon}\to \mathbb{R}$, and  $\calC^{\varepsilon}(\cdot,\cdot,\cdot):\calX^{\varepsilon} \times\calX^{\varepsilon} \times\calX^{\varepsilon} \to \mathbb{R}$, $\calL^{\varepsilon}(\cdot):\calX^{\varepsilon} \to \mathbb{R}$ by
\begin{align*}
  \calA^{\varepsilon} \left( (\uu,\emB),(\vv,\emC) \right)
  &\coloneqq  2 \int_{\Omega^{\varepsilon}_f} \DD(\uu):\DD(\vv) \di x
  \\
  &{}+{}
    \frac{\nAl}{\nRm} \left[ \int_{\Omega} \Div \emB \cdot \Div
      \emC \di x + \int_{\Omega^{\varepsilon}_s} \Curl \emB \cdot
    \Curl \emC \di x
    \right],\\
  \calB^{\varepsilon} \left( (\vv,\emC),p \right)
  &\coloneqq \int_{\Omega} p \Div \vv \di x,\\
  \calC^{\varepsilon} \left( (\uu_1,\emC_1), (\uu_2, \emC_2), (\uu_3, \emC_3) \right)
  &\coloneqq  \nRe\int_{\Omega} \left( \uu_1 \cdot \nabla \right) \uu_2
    \cdot \uu_3 \di x \\
  &
    {}-{} \nAl \int_{\Omega^{\varepsilon}_s} \left[ (\Curl \emC_2
    \times \emC_1) \cdot \uu_3 \right.\nonumber\\
  &\qquad{} +{}\left. \left( \uu_2 \times \emC_1
    \right)\cdot \Curl \emC_3  \right]\di x,\\
  \calL^{\varepsilon} (\vv,\emC)
  &\coloneqq \nRe \int_{\Omega} \bg \cdot \vv \di x+ \nAl\int_{\Omega^{\varepsilon}_s}
    \hh \cdot \emC \di x.
\end{align*}

We consider the weak formulation of problem \eqref{eq:5}:

\emph{Find $\left( (\uu^{\varepsilon},\emB^{\varepsilon}), p^{\varepsilon} \right) \in \calX^{\varepsilon} \times
  \calP^{\varepsilon}$ such that for all $\left( (\vv,\emC), q \right) \in
  \calX^{\varepsilon}\times \calP^{\varepsilon}$,}
\begin{align}
\label{eq:594}
  \begin{split}
\calA^{\varepsilon} \left( (\uu^{\varepsilon},\emB^{\varepsilon}),(\vv,\emC) \right) + \calB^{\varepsilon} \left( (\vv,\emC),p^{\varepsilon}
\right) + \calC^{\varepsilon} \left( (\uu^{\varepsilon},\emB^{\varepsilon}), (\uu^{\varepsilon}, \emB^{\varepsilon}), (\vv, \emC) \right)
&= \calL^{\varepsilon} (\vv,\emC),\\
\calB^{\varepsilon} \left( (\uu^{\varepsilon},\emB^{\varepsilon}),q \right)
&=0.
  \end{split}
\end{align}
Before showing that the weak formulation \eqref{eq:594} is equivalent
to the strong formulation \eqref{eq:5}, we recall:
\begin{lemma}[Lemma 3.17 \cite{gerbeauMathematicalMethodsMagnetohydrodynamics2006}]
\label{sec:stat-ferr-equat-1}
If $\emB \in H^1_n(\Omega,\RR^d)$, then there exists $\psi \in
H^2(\Omega)$ such that 
\begin{equation}
\label{eq:595}
\left\{
  \begin{array}{r l l}
-\Delta \psi & \displaystyle = \Div \emB &\text{ in }\Omega,\\[1pt]
 \displaystyle     \frac{\partial \psi}{\partial n}
      & \displaystyle = 0 &\text{ on }\partial \Omega.
  \end{array}
  \right.
\end{equation}
In particular, $\nabla \psi \in H^1_n(\Omega,\RR^d)$.
\end{lemma}

\begin{proposition}
\label{sec:stat-ferr-equat-2}
Suppose that $\hh$ satisfies \eqref{eq:596}. Then 
$\left( (\uu^{\varepsilon},\emB^{\varepsilon}),p^{\varepsilon} \right) \in \calX^{\varepsilon} \times \calP^{\varepsilon}$ is a
weak solution of \eqref{eq:594} if and only if it is a solution of \eqref{eq:5}-\eqref{eq:583s}. 
\end{proposition}

\begin{proof}
The incompressibility condition \eqref{eq:7} is straightforward from
the second equation of \eqref{eq:594}. We rewrite the first equation
of \eqref{eq:594} as
\begin{align}
\label{eq:593}
  \begin{split}
    &2\int_{\Omega^{\varepsilon}_f} \DD(\uu^{\varepsilon}):\DD(\vv) \di x +
    \frac{\nAl}{\nRm} \left[ \int_{\Omega} \Div \emB^{\varepsilon} \cdot \Div
      \emC \di x + \int_{\Omega^{\varepsilon}_s} \Curl
      \emB^{\varepsilon} \cdot \Curl \emC \di x
    \right] \\
    &{}+{} \int_{\Omega} p^{\varepsilon} \Div \vv \di x {}+{} \nRe \int_{\Omega} \left( \uu^{\varepsilon} \cdot \nabla \right) \uu^{\varepsilon}
    \cdot \vv \di x \\
    &{}-{} \nAl \int_{\Omega^{\varepsilon}_s} (\Curl \emB^{\varepsilon} \times
    \emB^{\varepsilon}) \cdot \vv \di x -\nAl\int_{\Omega^{\varepsilon}_s} \left( \uu^{\varepsilon} \times \emB^{\varepsilon}
    \right)\cdot \Curl \emC \di x \\
    &=  \nRe \int_{\Omega} \bg \cdot \vv \di x
    + \nAl \int_{\Omega^{\varepsilon}_s}\hh \cdot \emC \di x.
  \end{split}
\end{align}
Let $\emC = 0$ and choose $\vv \in C^{\infty}_c(\Omega^{\varepsilon}_f,\RR^d)$, then using
integration by parts, we obtain \eqref{eq:6}. Setting $\emC = 0$ again
and choosing $\vv \in H_0^1(\Omega, \RR^d)$ with $\DD(\vv) = 0$ on
$\Omega^{\varepsilon}_s$, we obtain the balance equations \eqref{eq:13}.

Next, choosing $\vv = 0$ in \eqref{eq:593} results in 
\begin{align}
  \label{eq:597}
  \begin{split}
\frac{\nAl}{\nRm} \left[ \int_{\Omega} \Div \emB^{\varepsilon} \cdot \Div
      \emC \di x + \int_{\Omega_s} \Curl \emB^{\varepsilon} \cdot
  \Curl \emC \di x
\right]\\
- \nAl\int_{\Omega_s} \left( \uu^{\varepsilon} \times \emB^{\varepsilon}
  \right)\cdot \Curl \emC \di x
  = \nAl \int_{\Omega_s} \hh \cdot \emC \di x.
  \end{split}
\end{align}
Let $\psi$ as in \cref{sec:stat-ferr-equat-1} and select $\emC =
\nabla \psi$ in \eqref{eq:597}, then by \eqref{eq:596},
\begin{align*}
-\frac{\nAl}{\nRm} \int_{\Omega} (\Div \emB^{\varepsilon})^2 \di x = 0
\end{align*}
so we obtain \eqref{eq:11}. Therefore, \eqref{eq:597} is simplified
to 
\begin{align*}
\frac{1}{\nRm} \int_{\Omega_s} \Curl \emB^{\varepsilon} \cdot \Curl \emC \di x
    - \int_{\Omega_s} \left( \uu^{\varepsilon} \times \emB^{\varepsilon}
  \right)\cdot \Curl \emC \di x
  = \int_{\Omega_s} \hh \cdot \emC \di x,
\end{align*}
Choose $\emC \in C^{\infty}_c(\Omega_s,\RR^d)$ and integrate by parts, this implies \eqref{eq:10}.

\end{proof}

\subsection{Existence and a priori estimates for the fine-scale velocity and the
magnetic field}
\label{sec:exist-priori-estim}


First, we recall
an important estimate for proving ellipticity \eqref{eq:600}:

\begin{proposition}[Theorem
  3.8 \cite{giraultFiniteElementMethods2012}]
\label{sec:stat-ferr-equat-4}
There exists $\kappa_{GR} > 0$ such that, for any $\emB \in \calH$, 
\begin{align}
\label{eq:603}
  \kappa_{GR} \norm{\emB}^2_{H^1(\Omega,\RR^d)}
  \le \norm{\Curl \emB}^2_{L^2(\Omega,\RR^d)} + \norm{\Div \emB}^2_{L^2(\Omega)}.
\end{align}
\end{proposition}

\begin{lemma}
\label{sec:stat-ferr-equat-5}
The form $\calA^{\varepsilon}$ is continuous and coercive on $\calX^{\varepsilon} \times
\calX^{\varepsilon}$, with coercivity constant $\alpha$ independent of
$\varepsilon$. In fact, $\alpha = \min \left\{
  \frac{\nAl}{\nRm}\kappa_{GR}, \kappa_K \right\} > 0$, where $\kappa_{GR}$
is the constant in \eqref{eq:603} and $\kappa_K^{-1}$ is the constant in
Korn's inequality.
\end{lemma}
\begin{proof}
For any $\left(\left( \uu, \emB \right), \left( \vv,\emC \right)\right)$ in $\calX^{\varepsilon} \times
\calX^{\varepsilon}$,  we have: 
\begin{align*}
\abs{\calA^{\varepsilon} \left( \left( \uu, \emB \right), \left( \vv,\emC \right)
  \right)}
  &\le 2\norm{\DD (\uu)}_{L^2(\Omega,\RR^{d \times d})} \norm{\DD
    (\vv)}_{L^2(\Omega,\RR^{d \times d})}\nonumber\\
  &\qquad{} +{} \frac{\nAl}{\nRm}
    \left[
    \norm{\Div \emB}_{L^2(\Omega)} \norm{\Div
    \emC}_{L^2(\Omega)}\right.\\
  & \qquad {}+{} \left.\norm{\Curl \emB}_{L^2(\Omega,\RR^d)} \norm{\Curl \emC}_{L^2(\Omega,\RR^d)}
    \right]\\
  &\le C\left( \Omega, \frac{\nAl}{\nRm}\right) \norm{(\uu,\emB)}_{\calX^{\varepsilon}} \norm{\left( \vv,\emC \right)}_{\calX^{\varepsilon}}.
\end{align*}
Therefore, $\calA^{\varepsilon}$ is continuous. Moreover, by
\eqref{eq:603} and Korn's inequality, 
\begin{align*}
  \calA^{\varepsilon} \left( (\uu,\emB), (\uu,\emB) \right)
  &\ge \int_{\Omega_f} \abs{\DD(\uu)}^2 \di x +
    \frac{\nAl}{\nRm} \left[
    \int_{\Omega} \abs{\Div \emB}^2 \di x + \int_{\Omega_s}
    \abs{\Curl \emB}^2 \di x
      \right]\\
  &=  \int_{\Omega} \abs{\DD(\uu)}^2 \di x +
    \frac{\nAl}{\nRm} \left[
    \int_{\Omega} \abs{\Div \emB}^2 \di x + \int_{\Omega}
    \abs{\Curl \emB}^2 \di x  
    \right]\\
  &\ge \alpha 
    \norm{(\uu,\emB)}^2_{\calX^{\varepsilon}}. 
\end{align*}
\end{proof}

\begin{lemma}
  \label{sec:stat-ferr-equat-7}
  The trilinear form $\calC^{\varepsilon}$ is continuous on $ \calX^{\varepsilon}
  \times \calX^{\varepsilon} \times \calX^{\varepsilon}$. Moreover, suppose
  $\rho_f=\rho_s$, then 
  for all $\left(
    (\uu,\emB), (\vv,\emC), (\ww,\emD) \right) \in \calX^{\varepsilon}
  \times \calX^{\varepsilon} \times \calX^{\varepsilon}$ with $\Div
  \uu = 0$, one has 
\begin{align*}
\calC^{\varepsilon} \left( (\uu, \emB), (\vv,\emC), (\ww, \emD) \right) = -\calC^{\varepsilon}
  \left( (\uu, \emB),  (\ww, \emD), (\vv,\emC) \right).
\end{align*}

\end{lemma} 

\begin{proof}
We write 
\begin{align*}
  &\abs{\calC^{\varepsilon} \left( (\uu, \emB), (\vv,\emC), (\ww, \emD) \right)}\\
  &\le C \left( \norm{\uu}_{H^1} \norm{\vv}_{H^1} \norm{\ww}_{H^1} +
    \norm{\emC}_{H^1} \norm{\emB}_{H^1}\norm{\ww}_{H^1} +
    \norm{\vv}_{H^1} \norm{\emB}_{H^1} \norm{\emD}_{H^1} \right)\\
  &\le C \norm{(\uu,\emB)}_{\calX^{\varepsilon}} \norm{(\vv,
    \emC)}_{\calX^{\varepsilon}} \norm{(\ww,\emD)}_{\calX^{\varepsilon}}.
\end{align*}

The second part is a consequence of the following identities:
\begin{align*}
  (\emB \times \Curl \emC) \cdot \vv
  &= (\vv \times \emB) \cdot \Curl \emC,\\
  \int_U \left( \uu \cdot \nabla\right)\vv \cdot \vv \di x
  &= -\frac{1}{2} \int_U \abs{\vv}^2 \Div \uu \di x + \frac{1}{2}
    \int_{\partial U} \abs{\vv}^2 \uu \cdot \nn \di \hmeas,  
\end{align*}
for $U=\Omega^{\varepsilon}_f$ or $U = \Omega^{\varepsilon}_s$.

Indeed, from the above identities and the definition of $\calC^{\varepsilon}$, one has $$\calC^{\varepsilon} \left(
  (\uu, \emB), (\vv,\emC), (\vv, \emC) \right) = 0$$ for all
$(\vv,\emC) \in \calX^{\varepsilon}$; therefore,
\begin{align}
  \label{eq:25}
  \begin{split}
    0 &= \calC^{\varepsilon} \left( (\uu, \emB), (\vv-\ww,\emC-\emD),
      (\vv-\ww, \emC-\emD)
    \right)\\
    &= \calC^{\varepsilon} \left( (\uu, \emB), (\vv,\emC), (\vv-\ww,
      \emC-\emD) \right) - \calC^{\varepsilon} \left( (\uu, \emB),
      (\ww,\emD), (\vv-\ww, \emC-\emD) \right)\\
    &= \calC^{\varepsilon} \left( (\uu, \emB), (\vv,\emC), (\vv, \emC)
    \right) - \calC^{\varepsilon} \left( (\uu, \emB), (\vv,\emC),
      (\ww, \emD) \right)
    \\
    &\quad {}-{} \left\{ \calC^{\varepsilon} \left( (\uu, \emB),
        (\ww,\emD), (\vv, \emC) \right) - \calC^{\varepsilon} \left(
        (\uu, \emB), (\ww,\emD), (\ww, \emD) \right)
    \right\}\\
    &= -\calC^{\varepsilon} \left( (\uu, \emB), (\vv,\emC), (\ww, \emD)
    \right) -  \calC^{\varepsilon} \left( (\uu, \emB), (\ww, \emD), (\vv,\emC) \right).
   \end{split}
\end{align}~
\end{proof}

We now define: 
\begin{align}
\label{eq:4}
  a^{\varepsilon} \left( (\uu, \emB); (\vv, \emC), (\ww, \emD) \right)
  &\coloneqq \calA^{\varepsilon} \left( (\vv,\emC),(\ww,\emD) \right)
    + \calC^{\varepsilon} \left( (\uu,\emB), (\vv,\emC), (\ww,\emD) \right).
\end{align}

\begin{lemma}
\label{sec:vari-form-1}
The following properties hold: 
\begin{itemize}
\item[(i)] For any $(\vv,\emC)$ in $\calY^{\varepsilon}$, we have: 
\begin{align}
\label{eq:15}
  a^{\varepsilon} \left( (\vv, \emC); (\vv, \emC), (\vv, \emC) \right)
  &\ge \alpha \norm{(\vv,\emC)}^2_{\calX^{\varepsilon}}.
\end{align}
Here $\alpha$ is the coercivity constant of $\calA^{\varepsilon}$ in
\cref{sec:stat-ferr-equat-5}.
\item[(ii)] If $(\uu_n,\emB_n)$ weakly converges to $(\uu,\emB)$ in
  $\calY^{\varepsilon}$, then for all $(\vv, \emC)$ in
  $\calX^{\varepsilon}$ we have: 
\begin{align}
\label{eq:16}
\lim_{n \to \infty} a^{\varepsilon} \left(
  (\uu_n,\emB_n);(\uu_n,\emB_n), (\vv,\emC) \right) = a^{\varepsilon}
  \left( (\uu,\emB); (\uu,\emB),(\vv,\emC) \right).
\end{align}
\item[(iii)]
  For all $(\uu_1, \emB_1), (\uu_2,\emB_2),
  (\vv,\emC)$ and $(\ww,\emD)$ in $\calX^{\varepsilon}$, we have: 
\begin{align}
  \label{eq:17}
  \begin{split}
  &\abs{
  a^{\varepsilon} \left( (\uu_1,\emB_1);(\vv,\emC), (\ww,\emD) \right)
  - a^{\varepsilon} \left( (\uu_2,\emB_2);(\vv,\emC), (\ww,\emD) \right)
  }\\
  &\le 
  \kappa_S \max\{1, 2 \nAl \} \norm{(\uu_1,\emB_1) -
    (\uu_2,\emB_2)}_{\calX^{\varepsilon}}
  \norm{(\vv,\emC)}_{\calX^{\varepsilon}} \norm{(\ww,\emD)}_{\calX^{\varepsilon}},
  \end{split}
\end{align}
where $\kappa_{S} = \kappa_S (d,\Omega)$ is the norm of the Sobolev embedding $H^1$ to $L^4$.
\end{itemize}
\end{lemma}
\begin{proof}
  
\begin{itemize}
\item[(i)]  This is a direct consequence of \eqref{eq:4},
  \cref{sec:stat-ferr-equat-5} and
  \cref{sec:stat-ferr-equat-7}.
\item[(ii)] Suppose $(\uu_n,\emB_n) \wcv (\uu,\emB)$ in
  $\calY^{\varepsilon}$.
  Write
\begin{align}
\label{eq:18}
  \begin{split}
    &\abs{a^{\varepsilon}\left( (\uu_n,\emB_n);(\uu_n,\emB_n),
        (\vv,\emC) - a^{\varepsilon}\left(
          (\uu,\emB);(\uu,\emB),(\vv,\emC) \right) \right)}\\
    &\quad \le \abs{\calA^{\varepsilon} \left( (\uu_n - \uu,\emB_n -
        \emB), (\vv,\emC) \right)}\\
  &\qquad {}+{} \abs{\calC^{\varepsilon}\left(
        (\uu_n-\uu,\emB_n - \emB), (\uu_n,\emB_n), (\vv,\emC) \right)}\\
    &\qquad{}+{} \abs{\calC^{\varepsilon} \left( (\uu,\emB), (\uu_n-\uu,
        \emB_n-\emB), (\vv,\emC) \right)}.
  \end{split}
\end{align}
Next, we have
\begin{align*}
&\calA^{\varepsilon}\left(\uu_n - \uu,\emB_n - \emB),(\vv,\emC) \right)\\
  &\quad =  2\int_{\Omega^{\varepsilon}_f} \DD(\uu_n-\uu):\DD(\vv) \di
    x \nonumber\\
  &\qquad
  {}+{}
    \frac{\nAl}{\nRm} \left[ \int_{\Omega} \Div (\emB_n-\emB) \cdot \Div
    \emC \di x + \int_{\Omega^{\varepsilon}_s} \Curl (\emB_n-\emB)
    \cdot \Curl \emC \di x
    \right]\\
  &\quad= 2\int_{\Omega} \DD(\uu_n-\uu):\DD(\vv) \di x\nonumber\\
  &\qquad
  {}+{}
    \frac{\nAl}{\nRm} \left[ \int_{\Omega} \Div (\emB_n-\emB) \cdot \Div
    \emC \di x + \int_{\Omega^{\varepsilon}_s} \Curl (\emB_n-\emB)
    \cdot \Curl \emC \di x
    \right],
\end{align*}
and thus, for each fixed $\varepsilon > 0$, the right-hand side
converges to 0 as $n \to \infty$.

For the second term on the right hand side of \eqref{eq:18}, we have by H\"older's inequality: 
\begin{align*}
&\abs{\calC^{\varepsilon}\left(
        (\uu_n-\uu,\emB_n - \emB), (\uu_n,\emB_n), (\vv,\emC) \right)}\\
  &\coloneqq  \left|\int_{\Omega} \left((\uu_n-\uu) \cdot \nabla \right) \uu_n
    \cdot \vv \di x \right. \\
  &\qquad
    {}-{} \nAl \left.\int_{\Omega^{\varepsilon}_s}\left[ (\Curl \emB_n \times (\emB_n - \emB)) \cdot \vv + \left( \uu_n \times (\emB_n - \emB)
    \right)\cdot \Curl \emC  \right]\di x\right|,\\
  &\le \norm{\uu_n-\uu}_{L^4} \norm{\nabla
    \uu_n}_{L^2} \norm{\vv}_{L^4} \\
  &\qquad{}+{} 2 \nAl
    \left[ \norm{\nabla \emB_n}_{L^2} \norm{\emB_n
    -\emB}_{L^4}\norm{\vv}_{L^4} + \norm{\uu_n}_{L^4} \norm{\emB_n - \emB}_{L^4}\norm{\Curl \emC}_{L^2} \right].
\end{align*}
By the Rellich–Kondrachov theorem, we have that, up to a subsequence,
$(\uu_n,\emB_n)$ strongly converges to $(\uu,\emB)$ in
$L^4(\Omega,\RR^d) \times L^4(\Omega,\RR^d)$. Therefore, the estimate
above shows that the second term on the right hand side of
\eqref{eq:18} also converges to 0 as $n \to  \infty$.

The last term on the right hand side of \eqref{eq:18} is: 
\begin{align*}
&\abs{\calC^{\varepsilon} \left( (\uu,\emB), (\uu_n-\uu,
        \emB_n-\emB), (\vv,\emC) \right)}
\coloneqq \left| \int_{\Omega} \left( \uu \cdot \nabla \right) (\uu_n-\uu)
    \cdot \vv \di x \right.\\
  &\qquad
    {}-{} \nAl \left. \int_{\Omega^{\varepsilon}_s} (\Curl
    (\emB_n-\emB) \times \emB) \cdot \vv \di x - \int_{\Omega^{\varepsilon}_s} \left( (\uu_n-\uu) \times \emB
    \right)\cdot \Curl \emC  \di x\right|,
\end{align*}
The first and the last integrals converge to 0 by a similar argument as
above. The middle one converges to 0 due to the weak convergence 
$\emB_n \wcv \emB$  in $H^1(\Omega,\RR^d)$.
\item[(iii)] From definition \eqref{eq:4} and the Sobolev embedding $H^1$
  to $L^4$, where the norm of the embedding is denoted by $\kappa_S(d,\Omega)$, we obtain
\begin{align*}
 &\abs{
  a^{\varepsilon} \left( (\uu_1,\emB_1);(\vv,\emC), (\ww,\emD) \right)
  - a^{\varepsilon} \left( (\uu_2,\emB_2);(\vv,\emC), (\ww,\emD) \right)
   }\\
  &= \abs{
    \calC^{\varepsilon} \left( (\uu_1,\emB_1);(\vv,\emC), (\ww,\emD)
    \right)
    -\calC^{\varepsilon}  \left( (\uu_2,\emB_2);(\vv,\emC), (\ww,\emD) \right)
    } \\
  &=\abs{
    \calC^{\varepsilon} \left( (\uu_1-\uu_2,\emB_1-\emB_2);(\vv,\emC), (\ww,\emD)
    \right)
    }\\
  &\le \norm{\uu_1-\uu_2}_{L^4} \norm{\nabla \vv}_{L^2}
    \norm{\ww}_{L^4} + 2\nAl \norm{\nabla \emC}_{L^2}
    \norm{\emB_1-\emB_2}_{L^4} \norm{\ww}_{L^4} \\
  & \qquad {}+{} 2 \nAl \mu_s \norm{\vv}_{L^4} \norm{\emB_1-\emB_2}_{L^4}
    \norm{\nabla \emD}_{L^2}\\
  &\le \kappa_S(d,\Omega) \max\{1, 2 \nAl \} \left[
    \norm{\uu_1-\uu_2}_{H^1} \norm{\nabla \vv}_{H^1}
    \norm{\ww}_{H^1}\right. \\
  &\qquad{}+{} \left. \norm{\nabla \emC}_{H^1}
    \norm{\emB_1-\emB_2}_{H^1} \norm{\ww}_{H^1}+\norm{\vv}_{H^1} \norm{\emB_1-\emB_2}_{H^1}
    \norm{\nabla \emD}_{H^1}
    \right]\\
  &\le  \kappa_S(d,\Omega) \max\{1, 2 \nAl \}
    \norm{(\uu_1,\emB_1) - (\uu_2,\emB_2)}_{\calX^{\varepsilon}}
    \norm{(\vv,\emC)}_{\calX^{\varepsilon}} \norm{(\ww,\emD)}_{\calX^{\varepsilon}}.
\end{align*}
  
\end{itemize}
\end{proof}


From \cref{sec:stat-ferr-equat-3}, \cref{sec:stat-ferr-equat-5},
\cref{sec:stat-ferr-equat-7} and \cref{sec:vari-form-1}, we conclude that
\begin{proposition}
  \label{sec:stat-ferr-equat-9}
  Let $\alpha=\min \left\{ \frac{\nAl}{\nRm}\kappa_{GR}, \kappa_K \right\}$
  be the coercivity constant of $\calA^{\varepsilon}$ in
  \cref{sec:stat-ferr-equat-5} and $\kappa_S = \kappa_S (d,\Omega)$ be
  the norm of the Sobolev embedding $H^1$ to $L^4$.  Then the variational
  problem \eqref{eq:594} has a solution
  $\left( (\uu^{\varepsilon}, \emB^{\varepsilon}),p^{\varepsilon}
  \right) \in \calX^{\varepsilon} \times \calP^{\varepsilon}$ such
  that:
\begin{align}
\label{eq:582}
  \norm{(\uu^{\varepsilon},\emB^{\varepsilon})}_{\calX^{\varepsilon}}
  \le \frac{\norm{\calL^{\varepsilon}}_{\left(\calY^{\varepsilon}\right)^{'}}}{\alpha}.
\end{align}
Moreover, if 
\begin{align}
\label{eq:19}
\kappa_S \max\{1, 2 \nAl \}
  \norm{\calL^{\varepsilon}}_{\left(\calY^{\varepsilon}\right)^{'}} \le \alpha^2,
\end{align}
then the solution is unique.
\end{proposition}

By H\"older's inequality: 
\begin{align*}
  \abs{\calL^{\varepsilon}(\vv,\emC)}
  &\le \nRe\norm{\bg}_{L^2} \norm{\vv}_{L^2} + \nAl
    \norm{\hh}_{L^2} \norm{\emC}_{L^2}\\
  &\le 2\left( \nRe \norm{\bg}_{L^2}+ \nAl
    \norm{\hh}_{L^2} \right) \norm{(\vv,\emC)}_{\calX^{\varepsilon}}.
\end{align*}

Thus, from \eqref{eq:582}, we obtain the following a priori estimate:
\begin{align}
\label{eq:20}
\norm{\left( \uu^{\varepsilon},\emB^{\varepsilon}  \right)}_{\calX^{\varepsilon}} \le
  \frac{2}{\alpha} \left( \nRe \norm{\bg}_{L^2}+ \nAl
    \norm{\hh}_{L^2} \right),
\end{align}
where the right-hand side is surely independent of $\varepsilon$.

\subsection{Existence and a priori estimate for the fine-scale pressure}
\label{sec:exist-priori-estim-2}

The following result is adapted from \cite[Theorem
  4.1]{acostaSolutionsDivergenceOperator2006} (see also \cite[Theorem 2.6]{acostaDivergenceOperatorRelated2017}, and \cite[Theorem III.3.1]{galdiIntroductionMathematicalTheory2011}),
\begin{theorem}
\label{sec:exist-priori-estim-3}
Let $\Omega \subset \RR^d$ be a Lipschitz domain with Lipschitz constant $\ell$. Then,
there exists a bounded linear operator $\Bog\colon L_0^2(\Omega) \to
H_0^1(\Omega,\RR^d)$, $f \mapsto \Bog f$, called the Bogovski\u{\i} map, such that, for all $f \in
L_0^2(\Omega)$, 
\begin{align}
\label{eq:46}
\Div \Bog f = f.
\end{align}
Moreover, the norm $\norm{\Bog}$ depends only on $d,\ell$ and $\diam (\Omega)$.
\end{theorem}

For $p \in \calP^{\varepsilon}$, there exists
$\vv \in \calV^{\varepsilon}$ such that $p = \Div \vv$. Thus $p = 0$
in $\Omega_s^{\varepsilon}$ since $\DD(\vv) = 0$ in
$\Omega_s^{\varepsilon}$. Adapting the construction in \cite[Step 1,
Proof of Lemma 3.3]{duerinckxQuantitativeHomogenizationTheory2021}
(see also \cite[Step 4, Proof of Proposition
2.1]{duerinckxCorrectorEquationsFluid2021}, \cite[Lemma
3.2]{duerinckxEffectiveViscosityRandom2021a}, \cite[Theorem
2.1]{bellaInverseDivergenceHomogenization2021}, \cite[Lemma 4.8]{hoferMotionSeveralSlender2021}) and using
\cref{sec:exist-priori-estim-3}, we obtain
\begin{lemma}
\label{sec:exist-priori-estim-4}
For each $p \in \calP^{\varepsilon}$, there exists $\vv \in
H_0^1(\Omega,\RR^d)$ such that 
\begin{enumerate}
\item $\vv$ is constant on $Y^{\varepsilon}_{i,s}$ for all $i$ (and
  thus $\vv \in \calV^{\varepsilon}$).
\item $\Div \vv = p$.
\item $\norm{\vv}_{H_0^1} \le \norm{\Bog} \norm{p}_{L^2}$.
\end{enumerate}
\end{lemma}
Note that we don't necessarily have $\vv = \Bog p$. Actually,  $\vv$ is obtained by modifying $\Bog p$ so that 1 and 3 are satisfied.

\begin{lemma}
\label{spacePe}
The space $\calP^\varepsilon$ defined in \cref{sec:main-result}
is a Hilbert space with respect to the $L^2-$inner product.
\end{lemma}

\begin{proof}
Let the space $\calP^{\varepsilon}$ be
equipped with the $L^2-$inner product. It is well-known that
$L_0^2(\Omega)$ is a Hilbert space with respect to this inner product (see
\cite[Lemma IV.1.9]{boyerMathematicalToolsStudy2013a}). Since
$\calP^{\varepsilon}$ is a subset of $L_0^2(\Omega)$ closed under addition and scalar multiplication, we only need
to show that $\calP^{\varepsilon}$ is closed. 
For that, let $\calP^{\varepsilon}
\ni q_n \to q_0 \in L_0^2(\Omega)$, we will prove that $q_0 \in
\calP^{\varepsilon}$.

Since $q_n \in \calP^{\varepsilon} = \Div (\calV^{\varepsilon}),$ by
\cref{sec:exist-priori-estim-4}, we
have $q_n = \Div \vv_n$, for some $\vv_n\in \calV^{\varepsilon}$ and
\begin{align*}
\norm{\vv_n}_{H_0^1(\Omega,\RR^d)}  \le
  \norm{\Bog} \norm{q_n}_{L_0^2(\Omega)}.
\end{align*}

Since $q_n$ converges to $q_0$ in $L_0^2(\Omega)$, it is bounded in
$L_0^2(\Omega)$, which implies that $\vv_n$ is also bounded in
$H_0^1(\Omega,\RR^d)$. On the one hand, since $H_0^1(\Omega,\RR^d)$ is
reflexive, the Eberlain--\v{S}mulian~theorem states that, up to a subsequence, there
exists a $\vv_0 \in H_0^1(\Omega,\RR^d)$ such that
$\vv_n \rightharpoonup \vv_0$ weakly in $H_0^1(\Omega,\RR^d)$. 
Testing this convergence with $ \calQ \in
C_c^{\infty}(\Omega,\RR^{d\times d})$, with $\supp \calQ \subset
\Omega_s^{\varepsilon}$ 
shows that $\vv_0 \in \calV^{\varepsilon}.$
On the other hand, 
by letting $\psi \in
C_c^{\infty}(\Omega)$, we observe that: 
 \begin{align*}
  \int_{\Omega} (q_0-\Div \vv_0) \psi \di x
  &= \int_{\Omega} (q_0 - q_n) \psi \di x + \int_{\Omega}(\Div \vv_n - \Div
    \vv_0) \psi \di x
  &\xrightarrow[]{n \to \infty} 0.
\end{align*}
Therefore, $q_0 = \Div \vv_0,$ which means that $q_0 \in \calP^{\varepsilon}$.
\end{proof}


\begin{lemma}
\label{sec:stat-ferr-equat-8}
The bilinear form $\calB^{\varepsilon}$ is continuous on $\calX^{\varepsilon}\times
\calP^{\varepsilon}$ and satisfies the inf-sup condition
\begin{align}
  \label{eq:40}
\exists \beta > 0 \quad \text{ such that } \quad
 \inf_{ q \in \calP^{\varepsilon}\setminus \left\{ 0 \right\}}
  \sup_{(\vv,\emC) \in \calX^{\varepsilon}\setminus \left\{ (\mathbf{0},\mathbf{0}) \right\}} \frac{\calB^{\varepsilon} \left(
  (\vv,\emC),q \right)}{\norm{(\vv,\emC)}_{\calX^{\varepsilon}}\norm{q}_{\calP^{\varepsilon}}} \ge
  \beta .
\end{align}
Moreover, the constant $\beta$ is independent of $\varepsilon$. In
particular, one can choose $\beta = \norm{\Bog}^{-1}$, where $\Bog$ is
the Bogovski\u{\i} map defined in \cref{sec:exist-priori-estim-3}.
\end{lemma}
\begin{proof}
Recall that $\calP^{\varepsilon}$ inherits the $L^2-$norm from $L^2_0(\Omega)$. We have:
\begin{align*}
\abs{\calB^{\varepsilon} \left(
  (\vv,\emC),q \right)} \le C \norm{q}_{L^2} \norm{\Div \vv}_{L^2} \le
  C \norm{q}_{\calP^{\varepsilon}} \norm{(\vv, \emC)}_{\calX^{\varepsilon}},
\end{align*}
so $\calB^{\varepsilon}$ is continuous on $\calX^{\varepsilon} \times \calP^{\varepsilon}$.

Since $\calP^{\varepsilon}$ is a Hilbert space by \cref{spacePe},
there exists a Riesz isomorphism
$\imath_R\colon \calP^{\varepsilon} \to \left( \calP^{\varepsilon}
\right)'$. Let $B \coloneqq \imath_R \circ \Div $ then $B$ is a
continuous surjective map from $\calV^{\varepsilon}$ to
$\left( \calP^{\varepsilon} \right)'$. Moreover, for
$\vv \in \calV^{\varepsilon}$ and $q \in \calP^{\varepsilon}$,
\begin{align}
  \label{eq:45}
  \dualpair{B\vv,p}{\left( \calP^{\varepsilon} \right)'}{\calV^{\varepsilon}}
  &= \dualpair{\imath_R \left( \Div \vv \right), q}{\left( \calP^{\varepsilon} \right)'}{\calV^{\varepsilon}} 
  = \left( \Div \vv, q \right)_{L^2} = \calB^{\varepsilon} \left(
  (\vv,\emC),q \right),
\end{align}
so $B$ is the operator associated to $\calB^{\varepsilon}$. Therefore,
the inf-sup condition follows by \cref{sec:check-infsup}.

Fix a function $q \in \calP^{\varepsilon}$, denote by $\vv_q$ the
corresponding field obtained from \cref{sec:exist-priori-estim-4}. We have
\begin{align}
  \label{eq:44}
 \begin{split}
&\sup_{(\vv,\emC) \in \calX^{\varepsilon}\setminus \left\{
  (\mathbf{0},\mathbf{0}) \right\}} \frac{\calB^{\varepsilon}\left(
  (\vv,\emC),q \right)}{\norm{(\vv,\emC)
  }_{\calX^{\varepsilon}}\norm{q}_{\calP^{\varepsilon}} }
  \ge \sup_{\substack{(\vv,\emC) \in \calX^{\varepsilon}\setminus \left\{
  (\mathbf{0},\mathbf{0}) \right\}\\ \emC = \mathbf{0}}} \frac{\calB^{\varepsilon}\left(
  (\vv,\emC),q \right)}{\norm{(\vv,\emC)}_{\calX^{\varepsilon}}\norm{q}_{\calP^{\varepsilon}}}\\
&\qquad= \sup_{\vv \in \calV^{\varepsilon}} \frac{\int_{\Omega} q \Div \vv
  \di x}{\norm{\vv}_{H^1_0}\norm{q}_{L^2}}
\ge \frac{\int_{\Omega} q \Div \vv_q
  \di x}{\norm{\vv_q}_{H^1_0}\norm{q}_{L^2}}
= \frac{\norm{q}_{L^2}}{\norm{\vv_q}_{H_0^1}}
\ge \frac{1}{\norm{\Bog}}.
\end{split}
\end{align}
Therefore, we choose $\beta = \norm{\Bog}^{-1}$, which is independent
of $\varepsilon$.
\end{proof}

\cref{sec:stat-ferr-equat-9} and \cref{sec:stat-ferr-equat-8} imply
the existence and uniqueness of the fine-scale pressure
$p^{\varepsilon}$. Moreover, from \eqref{eq:40} (with $\beta =
\norm{\Bog}^{-1}$), \eqref{eq:594}, \eqref{eq:20}, we have 
\begin{align*}
  \norm{p^{\varepsilon}}_{L^2}
  &\le \norm{\Bog}  \sup_{(\vv,\emC) \in \calX^{\varepsilon}\setminus \left\{ (\mathbf{0},\mathbf{0}) \right\}} \frac{\calB^{\varepsilon} \left(
    (\vv,\emC),q \right)}{\norm{(\vv,\emC)}_{\calX^{\varepsilon}}}\\
  &\le \norm{\Bog}  \sup_{(\vv,\emC) \in \calX^{\varepsilon}\setminus
    \left\{ (\mathbf{0},\mathbf{0}) \right\}}
    \frac{1}{\norm{(\vv,\emC)}_{\calX^{\varepsilon}}}
    \left\{
    \abs{\calA^{\varepsilon} \left(
    (\uu^{\varepsilon},\emB^{\varepsilon}),(\vv,\emC) \right)}
    \right.\\
  &\qquad{}+{}\left. \abs{\calC^{\varepsilon} \left( (\uu^{\varepsilon},\emB^{\varepsilon}), (\uu^{\varepsilon}, \emB^{\varepsilon}), (\vv, \emC) \right)}
+ \abs{ \calL^{\varepsilon}
    (\vv,\emC)}
    \right\}
  \\
  &\le C \left( \norm{\left( \uu^{\varepsilon},\emB^{\varepsilon}
    \right)}_{\calX^{\varepsilon}} +  \norm{\left(
    \uu^{\varepsilon},\emB^{\varepsilon}
    \right)}_{\calX^{\varepsilon}}^2 +  \nRe \norm{\bg}_{L^2}+ \nAl
    \norm{\hh}_{L^2}\right).
\end{align*}
In particular, by \eqref{eq:20}, we obtain
\begin{align}
\label{eq:47}
  \norm{p^{\varepsilon}}_{L^2}
  \le C \left( \nRe \norm{\bg}_{L^2}+ \nAl
    \norm{\hh}_{L^2} + 1\right)^2,
\end{align}
where $C$ is independent of $\varepsilon$.

\subsection{The two-scale homogenized problem}
\label{sec:two-scale-homog}

By \eqref{eq:20} and \eqref{eq:47}, there exist
$\uu^0 \in H^1_0(\Omega,\RR^d),$ $\emB^0 \in H_n^1(\Omega,\RR^d)$,
$\uu^1 \in L^2 \left( \Omega, H^1_{\per}(Y,\RR^d)/\RR \right),$
$\emB^1 \in L^2 \left( \Omega, H^1_{\per}(Y,\RR^d)/\RR \right)$ and
$p^0 \in L^2_0 (\Omega \times Y)$ such
that, up to a subsequence:
\begin{align}
\label{eq:21}
  \begin{split}
    &\uu^{\varepsilon} \wcv \uu^0,~\emB^{\varepsilon} \wcv \emB^0
    \text{ weakly in } H^1(\Omega,\RR^d),\\
    &\uu^{\varepsilon} \tscale \uu^0,~ \emB^{\varepsilon} \tscale
    \emB^0 \text{ two-scale},\\
    &\nabla \uu^{\varepsilon} \tscale \nabla \uu^0(x) + \nabla_y
    \uu^1(x,y),~ \nabla \emB^{\varepsilon} \tscale \nabla \emB^0(x) +
    \nabla_y \emB^1(x,y) \text{ two-scale},\\
    &p^{\varepsilon} \wcv[2] p^0 \text{ two-scale}.
  \end{split}
\end{align}

Let
$\vv = \vv^0 (\cdot) + \varepsilon \vv^1 \left(
  \cdot,\frac{\cdot}{\varepsilon} \right)$ and
$\emC = \emC^0 (\cdot) + \varepsilon \emC^1 \left(
  \cdot,\frac{\cdot}{\varepsilon} \right),$ with
$\vv^0, \emC^0 \in \calD (\Omega,\RR^d)$ and
$\vv^1,\emC^1 \in \calD \left( \Omega, C^{\infty}_{\per}(Y,\RR^d)
\right)$. Let
$q = q^0(\cdot) + \varepsilon q^1 \left(
  \cdot,\frac{\cdot}{\varepsilon} \right)$ with
$q^0 \in \calD(\Omega)$ and
$q^1 \in \calD \left( \Omega, C_{\per}^{\infty}(Y) \right)$.
\paragraph{The effective form corresponding to $\calA^\varepsilon$}~\\

By definition and \eqref{eq:577-nd},
\begin{align*}
  \calA^{\varepsilon}
  &\left( (\uu^{\varepsilon},\emB^{\varepsilon}),(\vv,\emC) \right)\\
  &=  2\int_{\Omega^{\varepsilon}_f} \DD(\uu^{\varepsilon}):\DD(\vv) \di x
  {}+{}
    \frac{\nAl}{\nRm} \left( \int_{\Omega} \Div \emB^{\varepsilon} \cdot \Div
    \emC \di x + \int_{\Omega^{\varepsilon}_s} \Curl
    \emB^{\varepsilon} \cdot \Curl \emC \di x
    \right)\\
  &=  2\int_{\Omega} \DD(\uu^{\varepsilon}):\DD(\vv) \di x
  {}+{}
    \frac{\nAl}{\nRm} \left( \int_{\Omega} \Div \emB^{\varepsilon} \cdot \Div
    \emC \di x + \int_{\Omega^{\varepsilon}_s} \Curl
    \emB^{\varepsilon} \cdot \Curl \emC \di x
    \right)\\
  &\eqqcolon 2Q_1 + \frac{\nAl}{\nRm} \left( Q_2 + Q_3 \right)
\end{align*}
Then \eqref{eq:21} implies
\begin{align}
\label{eq:22}
  \begin{split}
    \lim_{\varepsilon \to 0}Q_1
    &=\lim_{\varepsilon \to 0} \int_{\Omega}\DD(\uu^{\varepsilon}): \DD(\vv)\di x\\
    &= \lim_{\varepsilon \to 0} \int_{\Omega}\DD(\uu^{\varepsilon}): \left[ \DD
      (\vv^0)(x) + \varepsilon \DD \left( \vv^1 \right) \left( x,
        \frac{x}{\varepsilon} \right) +  \DD_y (\vv^1)\left( x, \frac{x}{\varepsilon}
      \right) \right]\di x \\
    &= \frac{1}{\abs{Y}} \int_{\Omega} \int_Y \left[ \DD(\uu^0) +
      \DD_y(\uu^1) \right]: \left[ \DD(\vv^0) + \DD_y (\vv^1)
    \right]\di y \di x.
  \end{split}
\end{align}
Similarly, we have
\begin{align}
\label{eq:23}
  \begin{split}
    \lim_{\varepsilon \to 0} Q_2
    &=
    \lim_{\varepsilon\to 0} \int_{\Omega} \Div \emB^{\varepsilon}
    \cdot \Div \emC \di x\\
    &= \frac{1}{\abs{Y}} \int_{\Omega} \int_Y \left( \Div \emB^0 +
      \Div_y \emB^1 \right) \cdot \left( \Div \emC^0 + \Div_y \emC^1
    \right) \di y \di x.
  \end{split}
\end{align}

To compute the limit of the integral $Q_3$, we make use of the following limiting
behaviors of the domain $\Omega_s^{\varepsilon}$, which varies as
$\varepsilon$ goes to 0. Clearly,
\begin{align}
\label{eq:30}
\mathds{1}_{\Omega^{\varepsilon}_s} \tscale
  \mathds{1}_{\Omega\times Y_s} \quad \text{ and }\quad
  \lim_{\varepsilon\to 0}
  \norm{\mathds{1}_{\Omega_s^{\varepsilon}}}_{L^2(\Omega)} =
  \norm{\mathds{1}_{\Omega \times Y_s}}_{L^2(\Omega \times Y)}.
\end{align}
Since $\mathds{1}_{\Omega \times Y_s} \in L^2_{\per} \left( Y,
  C(\Omega) \right)$, we obtain from \cref{sec:two-scale-corrector}
that 
\begin{align}
\label{eq:38}
\lim_{\varepsilon \to 0} \norm{\mathds{1}_{\Omega^{\varepsilon}_s}(x) -
  \mathds{1}_{\Omega \times Y_s} \left( x, \frac{x}{\varepsilon}
  \right)}_{L^2(\Omega)} = 0.
\end{align}
Now we write 
\begin{align*}
  Q_3
  &= \int_{\Omega_s^{\varepsilon}} \Curl \emB^{\varepsilon}(x) \cdot
    \Curl \emC^0 (x) \di x + \varepsilon \int_{\Omega_s^{\varepsilon}}
    \Curl \emB^{\varepsilon}(x)\cdot \Curl_y \emC^1 \left( x,
    \frac{x}{\varepsilon} \right) \di x \\
  &{}\qquad {}+ \int_{\Omega_s^{\varepsilon}} \Curl \emB^{\varepsilon}
    \cdot \Curl_y \emC^1 \left( x, \frac{x}{\varepsilon} \right) \di
    x\\
  &\eqqcolon L_1 + L_2 + L_3.
\end{align*}
Clearly, $\lim_{\varepsilon \to 0} L_2 = 0.$
By \eqref{eq:30}, \eqref{eq:21} and \eqref{eq:36} of \cref{sec:two-scale-corrector} we have 
\begin{align*}
  \lim_{\varepsilon \to 0} L_1
  &= \lim_{\varepsilon \to 0} \int_{\Omega} \mathds{1}_{\Omega_s^{\varepsilon}}(x) \Curl
    \emB^{\varepsilon}(x) \cdot \Curl \emC^0(x) \di x \\
  &= \frac{1}{\abs{Y}} \int_{\Omega} \int_{Y_s} \left( \Curl \emB^0(x) +
    \Curl_y \emB^1(x,y) \right) \cdot \Curl \emC^0(x) \di y \di x.
\end{align*}
And finally, for $L_3$, we have
\begin{align*}
  L_3
  &= \int_{\Omega} \Curl \emB^{\varepsilon}(x) \cdot \left(
    \mathds{1}_{\Omega_s^{\varepsilon}}(x) -\mathds{1}_{\Omega \times
    Y_s} \left( x, \frac{x}{\varepsilon} \right) \right) \Curl_y \emC^1
    \left( x, \frac{x}{\varepsilon} \right) \di x\\
  &{}\quad{}+{} \int_{\Omega} \Curl \emB^{\varepsilon}(x) \cdot
    \mathds{1}_{\Omega \times Y_s} \left( x, \frac{x}{\varepsilon}
    \right)\Curl_y \emC^1 \left( x, \frac{x}{\varepsilon} \right) \di x,
\end{align*}
For the first integral above, we obtain
\begin{align*}
&\abs{\int_{\Omega} \Curl \emB^{\varepsilon}(x) \cdot \left(
    \mathds{1}_{\Omega_s^{\varepsilon}}(x) -\mathds{1}_{\Omega \times
    Y} \left( x, \frac{x}{\varepsilon} \right) \right) \Curl_y \emC^1
                 \left( x, \frac{x}{\varepsilon} \right) \di x}\\
  &\quad \le C \norm{\nabla_y \emC^1}_{L^{\infty}} \norm{\nabla
    \emB^{\varepsilon}}_{L^2}
    \norm{\mathds{1}_{\Omega_s^{\varepsilon}} (x) - \mathds{1}_{\Omega
    \times Y_s} \left( x, \frac{x}{\varepsilon} \right)}_{L^2} \cv 0
\end{align*}
as $\varepsilon \to 0$ due to \eqref{eq:38} and H\"older's inequality. By the latter and \eqref{eq:21}, we have
\begin{align*}
  \lim_{\varepsilon \to 0} L_3
  &= \lim_{\varepsilon \to 0 } \int_{\Omega} \Curl
  \emB^{\varepsilon}(x) \cdot \mathds{1}_{\Omega \times Y_s} \Curl_y
    \emC^1 \left( x, \frac{x}{\varepsilon} \right) \di x\\
  &= \frac{1}{\abs{Y}} \int_{\Omega} \int_{Y_s} \left( \Curl \emB^0(x)
    + \Curl_y \emB^1 (x,y)\right) \cdot \Curl_y \emC^1 (x,y) \di y \di x.
\end{align*}
In conclusion, we have
\begin{align}
\label{eq:24}
\lim_{\varepsilon \to 0} Q_3
  &= \frac{1}{\abs{Y}} \int_{\Omega} \int_{Y_s} \left( \Curl \emB^0 +
    \Curl_y \emB^1 \right) \cdot \left( \Curl \emC^0 + \Curl_y \emC^1
    \right) \di y \di x.
\end{align}

From \eqref{eq:22}, \eqref{eq:23} and \eqref{eq:24}, the effective form $\calA^0$, corresponding to the limit as $\varepsilon\to0$ of $\calA^{\varepsilon}$, is given by
\begin{align}
\label{eq:26}
  \begin{split}
    \calA^0
    &\coloneqq \frac{2}{\abs{Y}} \int_{\Omega} \int_Y \left( \DD(\uu^0) +
      \DD_y(\uu^1) \right): \left[ \DD(\vv^0) + \DD_y (\vv^1)
    \right]\di y \di x \\
    &\quad {}+{}
    \frac{\nAl}{\nRm} \left\{
      \frac{1}{\abs{Y}} \int_{\Omega} \int_Y \left( \Div \emB^0 +
      \Div_y \emB^1 \right) \cdot \left( \Div \emC^0 + \Div_y \emC^1
    \right) \di y \di x \right.\\
  &\qquad + \left.
    \frac{1}{\abs{Y}} \int_{\Omega} \int_{Y_s} \left( \Curl \emB^0 +
    \Curl_y \emB^1 \right) \cdot \left( \Curl \emC^0 + \Curl_y \emC^1
    \right) \di y \di x
    \right\}.
  \end{split}
\end{align}

\paragraph{The effective forms corresponding to $\calB^{\varepsilon}$ and $\calL^{\varepsilon}$}~\\

From the last limit of \eqref{eq:21}, we have 
\begin{align}
  \label{eq:48}
  \begin{split}
  \lim_{\varepsilon\to 0} \calB^{\varepsilon}\left(
  (\vv,\emC),p^{\varepsilon} \right)
  &=\lim_{\varepsilon \to 0} \int_{\Omega} p^{\varepsilon} \Div \vv
    \di x\\
  &=\lim_{\varepsilon \to 0} \int_{\Omega} p^{\varepsilon}(x) \left(
    \Div \vv^0(x) + \varepsilon \Div \vv^1 \left( x,
    \frac{x}{\varepsilon} \right) + \Div_y \vv^1 \left( x,
    \frac{x}{\varepsilon} \right) \right) \di x\\
  &= \frac{1}{\abs{Y}} \int_{\Omega} \int_Y p^0 \left( \Div \vv^0 +
    \Div_y \vv^1 \right) \di y \di x.
  \end{split}
\end{align}

Moreover, 
\begin{align}
\label{eq:52}
  \begin{split}
    \lim_{\varepsilon \to 0} \calB^{\varepsilon} \left( \left(
        \uu^{\varepsilon}, \emB^{\varepsilon} \right), q \right)
    &= \lim_{\varepsilon \to 0} \int_{\Omega} q \Div \uu^{\varepsilon}
    \di x\\
    &= \frac{1}{\abs{Y}} \int_{\Omega} \int_Y
    q^0(x) \left( \Div \uu^0(x) + \Div_y \uu^1(x,y) \right) \di y
    \di x.
  \end{split}
\end{align}
From \eqref{eq:30}, we have 
\begin{align}
  \label{eq:49}
  \begin{split}
  \lim_{\varepsilon \to 0} \calL^{\varepsilon} (\vv,\emC)
  &= \lim_{\varepsilon \to 0} \left( \nRe \int_{\Omega} \bg \cdot \vv \di x+ \nAl\int_{\Omega^{\varepsilon}_s} \hh \cdot \emC \di x \right)\\
  &= \nRe  \int_{\Omega}  \bg \cdot \vv^0 \di x + \nAl
    \frac{\abs{Y_s}}{\abs{Y}} \int_{\Omega} \hh \cdot \emC^0 \di x.
    \end{split}
\end{align}
\paragraph{The effective form corresponding to $\calC^{\varepsilon}$}~\\

Recall that
\begin{align*}
\calC^{\varepsilon} &\left((\uu^{\varepsilon},\emB^{\varepsilon}), (\uu^{\varepsilon},\emB^{\varepsilon}),(\vv,\emC) \right)\\  
  &=\nRe \int_{\Omega} \left( \uu^{\varepsilon} \cdot \nabla \right) \uu^{\varepsilon}
    \cdot \vv \di x \\
  &\qquad{}-{} \nAl \int_{\Omega^{\varepsilon}_s} (\Curl \emB^{\varepsilon} \times
    \emB^{\varepsilon}) \cdot \vv \di x -\nAl\int_{\Omega^{\varepsilon}_s} \left( \uu^{\varepsilon} \times \emB^{\varepsilon}
    \right)\cdot \Curl \emC \di x\\
  &\eqqcolon
    \nRe I_1 -\nAl I_2 - \nAl I_3.
\end{align*}

\begin{itemize}
\item To obtain $\lim_{\varepsilon \to 0} I_1$, we split 
\begin{align}
  \label{eq:28}
  \begin{split}
  I_1
  &=  \int_{\Omega} \left( \uu^{\varepsilon} \cdot \nabla \right)
    \uu^{\varepsilon} \cdot \vv \di x \\
  &=\int_{\Omega} \left( (\uu^{\varepsilon}-\uu^0) \cdot \nabla \right)
    \uu^{\varepsilon} \cdot \vv \di x + \int_{\Omega} \left( \uu^0
      \cdot \nabla \right) \uu^{\varepsilon} \cdot \vv \di x\\
  &\eqqcolon J_1 +  J_2.
  \end{split}
\end{align}
From \eqref{eq:20} and \eqref{eq:21}, we have
\begin{align*}
  \lim_{\varepsilon \to 0}\abs{J_1}
  &\le \lim_{\varepsilon \to 0}  \norm{\uu^{\varepsilon} -\uu^0}_{L^2}
    \norm{\nabla \uu^{\varepsilon}}_{L^2} \norm{\vv}_{L^{\infty}}\\
  &\le  \lim_{\varepsilon \to 0}  \norm{\uu^{\varepsilon} - \uu^0}_{L^2}
    \frac{1}{\alpha}\left( \nRe \norm{\bg}_{L^2}+ \nAl\norm{\hh}_{L^2} \right)
    \norm{\vv}_{L^{\infty}} = 0.
\end{align*}
From the above, \eqref{eq:21} and since $\left( \uu^{\varepsilon} \cdot
  \nabla \right) \uu^{\varepsilon} = u^{\varepsilon}_i
\frac{\partial}{\partial x_i} \uu^{\varepsilon} $, we obtain
\begin{align}
  \label{eq:27}
  \begin{split}
  \lim_{\varepsilon \to 0} I_1=\lim_{\varepsilon \to 0} J_2
  &= \lim_{\varepsilon \to 0}  \int_{\Omega} u^{\varepsilon}_i(x)
    \frac{\partial}{\partial x_i} \uu^{\varepsilon}(x) \cdot \left(
      \vv^0(x) + \varepsilon \vv^1 \left( x, \frac{x}{\varepsilon} \right) \right) \di x\\
  &=  \frac{1}{\abs{Y}} \int_{\Omega} \int_Y u^0_i(x) \left(
    \frac{\partial}{\partial x_i} \uu^0(x) + \frac{\partial}{\partial
      y_i} \uu^1(x,y) \right) \cdot \vv^0(x)\di y \di x.
  \end{split}
\end{align}
\item Similarly, to obtain $\lim_{\varepsilon \to 0}  I_2$, we split
\begin{align*}
  I_2
  &=  \int_{\Omega^{\varepsilon}_s} \left( \Curl
    \emB^{\varepsilon} \times \emB^{\varepsilon} \right) \cdot \vv
    \di x \\
  &=  \int_{\Omega_s^{\varepsilon}} \left[ \Curl
    \emB^{\varepsilon} \times \left( \emB^{\varepsilon} - \emB^0 \right)
    \right] \cdot \vv \di x +  \int_{\Omega_s^{\varepsilon}} \left(
    \Curl \emB^{\varepsilon} \times \emB^0 \right) \cdot \vv
    \di x\\
  &\eqqcolon  K_1 +  K_2.
\end{align*}
From \eqref{eq:20} and \eqref{eq:21}, we have 
\begin{align*}
  \lim_{\varepsilon \to 0} \abs{K_1}
  &= \lim_{\varepsilon \to 0}  \norm{\Curl
    \emB^{\varepsilon}}_{L^2} \norm{\emB^{\varepsilon} - \emB^0}_{L^2}
    \norm{\vv}_{L^{\infty}} \\
  &\le \lim_{\varepsilon \to 0} 2
    \norm{\nabla\emB^{\varepsilon}}_{L^2} \norm{\emB^{\varepsilon} - \emB^0}_{L^2}
    \norm{\vv}_{L^{\infty}} \\
  &\le\lim_{\varepsilon \to 0}  2 \frac{1}{\alpha}\left( \nRe
    \norm{\bg}_{L^2}+ \nAl\norm{\hh}_{L^2} \right)
    \norm{\emB^{\varepsilon}-\emB^0}_{L^2} \norm{\vv}_{L^{\infty}} = 0.
\end{align*}
From the above, \eqref{eq:21}, \eqref{eq:30} and \eqref{eq:36} of \cref{sec:two-scale-corrector}, we obtain 
\begin{align*}
   \lim_{\varepsilon \to 0} I_2 &= \lim_{\varepsilon \to 0} K_2\\
  &= \lim_{\varepsilon \to 0} \int_{\Omega}
    \mathds{1}_{\Omega_s^{\varepsilon}}(x) \left( \Curl
    \emB^{\varepsilon}(x)\times \emB^0(x) \right)\cdot \left( \vv^0(x) +
    \varepsilon \vv^1 \left( x, \frac{x}{\varepsilon} \right) \right)
    \di x\\
  &= \frac{1}{\abs{Y}} \int_{\Omega} \int_{Y_s} \left( \Curl
    \emB^0(x) + \Curl_y \emB^1(x,y) \right) \times \emB^0(x) \cdot
    \vv^0(x) \di y \di x.
\end{align*}
\item Finally, to obtain $ \lim_{\varepsilon \to 0}  I_3$: \begin{align*}
  I_3
  &= \int_{\Omega_s^{\varepsilon}} \left( \uu^{\varepsilon}(x)
    \times \emB^{\varepsilon}(x) \right) \cdot \left( \Curl \emC^0(x) +
    \varepsilon \Curl \emC^1 \left( x,\frac{x}{\varepsilon} \right)\right.\\
  &\qquad{} +{}\left.
    \Curl_y \emC^1 \left( x,\frac{x}{\varepsilon} \right) \right) \di
    x\\
  &= \int_{\Omega}\mathds{1}_{\Omega_s^{\varepsilon}} \left( \uu^{\varepsilon}(x)
    \times \emB^{\varepsilon}(x) \right) \cdot  \Curl \emC^0(x)\di
    x\\
   &\qquad {}+{} \int_{\Omega}\mathds{1}_{\Omega_s^{\varepsilon}} \left( \uu^{\varepsilon}(x)
    \times \emB^{\varepsilon}(x) \right) \cdot 
    \Curl_y \emC^1 \left( x,\frac{x}{\varepsilon} \right) \di
    x
  \\
  &\eqqcolon M_1 + M_2.
\end{align*}

By the Rellich–Kondrachov theorem, we have $\uu^{\varepsilon}$
and $\emB^{\varepsilon}$ strongly converge to $\uu^0$ and $\emB^0$ in
$L^4(\Omega,\RR^d)$, respectively. Moreover, since $\emC^0$ is smooth, we have
$\left( \uu^{\varepsilon} \times \emB^{\varepsilon} \right) \cdot
\Curl \emC^0$ strongly converges to
$\left( \uu^0 \times \emB^0 \right) \cdot \Curl \emC^0$ in $L^2$. Also
$\mathds{1}_{\Omega_s^{\varepsilon}} \wcv \frac{1}{\abs{Y}} \int_Y
\mathds{1}_{\Omega \times Y_s} \di y$ in $L^2$, so
\begin{align*}
  \lim_{\varepsilon \to 0} M_1
  &=  \frac{1}{\abs{Y}} \int_{\Omega} \int_{Y_s} \left( \uu^0(x)\times
    \emB^0(x) \right) \cdot \Curl \emC^0(x)
     \di y \di x.
\end{align*}
Next, rewrite $M_2$ as
\begin{align*}
  M_2
  &=\int_{\Omega}\mathds{1}_{\Omega_s^{\varepsilon}} \left( \uu^{\varepsilon}(x)
    \times \emB^{\varepsilon}(x) - \uu^0(x) \times \emB^0(x)\right) \cdot 
    \Curl_y \emC^1 \left( x,\frac{x}{\varepsilon} \right) \di x\\
  &{}\qquad {}+{} \int_{\Omega}\mathds{1}_{\Omega_s^{\varepsilon}} \left(\uu^0(x) \times \emB^0(x)\right) \cdot 
    \Curl_y \emC^1 \left( x,\frac{x}{\varepsilon} \right) \di x
\end{align*}
Since $\uu^{\varepsilon} \times \emB^{\varepsilon}$ strongly converges
to $\uu^0 \times \emB^0$ in $L^2$, we have 
\begin{align*}
&\abs{\int_{\Omega}\mathds{1}_{\Omega_s^{\varepsilon}} \left( \uu^{\varepsilon}(x)
    \times \emB^{\varepsilon}(x) - \uu^0(x) \times \emB^0(x)\right) \cdot 
  \Curl_y \emC^1 \left( x,\frac{x}{\varepsilon} \right) \di x
  }\\
  &\qquad\le \norm{\uu^{\varepsilon} \times \emB^{\varepsilon}-\uu^0 \times
    \emB^0}_{L^2} \norm{\Curl_y \emC^1}_{L^{\infty}} \cv 0 \quad \text{as} \quad
    \varepsilon \to 0.
\end{align*}
Thus, using $\mathds{1}_{\Omega_s^{\varepsilon}} \tscale
\mathds{1}_{\Omega \times Y_s},$ we obtain
\begin{align*}
  \lim_{\varepsilon \to 0} M_2
  &=  \frac{1}{\abs{Y}} \int_{\Omega} \int_{Y_s} \left( \uu^0(x)\times
    \emB^0(x) \right) \cdot   \Curl_y \emC^1(x,y)
     \di y \di x.
\end{align*}
Therefore, 
\begin{align}
\label{eq:39}
  \lim_{\varepsilon \to 0} I_3
  &= \frac{1}{\abs{Y}} \int_{\Omega} \int_{Y_s} \left( \uu^0(x)\times
    \emB^0(x) \right) \cdot \left( \Curl \emC^0(x) + \Curl_y \emC^1(x,y)
    \right) \di y \di x.
\end{align}
\end{itemize}

\paragraph{Summary}~\\
We now collect all relevant results obtained above in order to derive the two-scale homogenized system.  In the weak formulation \eqref{eq:594}, we 
choose
$\vv = \vv^0 (\cdot) + \varepsilon \vv^1 \left(
  \cdot,\frac{\cdot}{\varepsilon} \right)$,
$\emC = \emC^0 (\cdot) + \varepsilon \emC^1 \left(
  \cdot,\frac{\cdot}{\varepsilon} \right),$ and $q = q^0 \left( \cdot
\right) + q^1 \left( \cdot, \frac{\cdot}{\varepsilon} \right)$, with
$\vv^0, \emC^0 \in \calD (\Omega,\RR^d)$, $q^0 \in \calD(\Omega)$ and
$\vv^1,\emC^1 \in \calD \left( \Omega, C^{\infty}_{\per}(Y,\RR^d)
\right),$ $q^1 \in \calD \left( \Omega, C_{\per}^{\infty}(Y)
\right)$. Then, letting $\varepsilon \to 0$, we obtain
\begin{align}
\label{eq:50}
    &\frac{2}{\abs{Y}} \int_{\Omega} \int_Y \left( \DD(\uu^0(x)) +
      \DD_y(\uu^1(x,y)) \right): \left( \DD(\vv^0(x)) + \DD_y
      (\vv^1(x,y))
    \right)\di y \di x \\ \nonumber
    & {}+{} \frac{\nAl}{\nRm} \left\{ \frac{1}{\abs{Y}}
      \int_{\Omega} \int_Y \left( \Div \emB^0(x) + \Div_y \emB^1(x,y)
      \right) \cdot \left( \Div \emC^0(x) + \Div_y \emC^1(x,y)
      \right) \di y \di x \right.\\ \nonumber
    & + \left.  \frac{1}{\abs{Y}} \int_{\Omega} \int_{Y_s}
      \left( \Curl \emB^0(x) + \Curl_y \emB^1(x,y) \right) \cdot
      \left( \Curl \emC^0(x) + \Curl_y \emC^1(x,y) \right) \di y \di x
    \right\}\\  \nonumber
    &{}+{} \frac{1}{\abs{Y}} \int_{\Omega} \int_Y p^0(x,y) \left( \Div
      \vv^0(x) +
      \Div_y \vv^1 (x,y)\right) \di y \di x\\ \nonumber
    &{}+{} \nRe \frac{1}{\abs{Y}} \int_{\Omega} \int_Y u^0_i(x) \left(
      \frac{\partial}{\partial x_i} \uu^0(x) +
      \frac{\partial}{\partial
        y_i} \uu^1(x,y) \right) \cdot \vv^0(x)\di y \di x\\ \nonumber
    & - \nAl \frac{1}{\abs{Y}} \int_{\Omega} \int_{Y_s} \left(
      \Curl \emB^0(x) + \Curl_y \emB^1(x,y) \right) \times \emB^0(x)
    \cdot
    \vv^0(x) \di y \di x \\ \nonumber
    & -\nAl \frac{1}{\abs{Y}} \int_{\Omega} \int_{Y_s} \left(
      \uu^0(x)\times \emB^0(x) \right) \cdot \left( \Curl \emC^0(x) +
      \Curl_y \emC^1(x,y)
    \right) \di y \di x\\
    &= \nRe \int_{\Omega} \bg (x)\cdot \vv^0 (x)\di x + \nAl
    \frac{\abs{Y_s}}{\abs{Y}} \int_{\Omega} \hh (x) \cdot \emC^0(x) \di x,\nonumber
\end{align}
and 
\begin{align}
\label{eq:51}
 \frac{1}{\abs{Y}} \int_{\Omega} \int_Y
    q^0(x) \left( \Div \uu^0(x) + \Div_y \uu^1(x,y) \right) \di y
  \di x
   =0.
\end{align}
Finally, testing \eqref{eq:7}, \eqref{eq:8}, \eqref{eq:9} and \eqref{eq:11}
with suitable test functions and applying \eqref{eq:21}, we obtain
\begin{align}
\label{eq:57}
\begin{split}
    \Div \uu^0 = 0 \text{ in }\Omega,\\
    ~\\
    \Div \emB^0 = 0 \text{ in }\Omega,\\
    ~
\end{split}
  \begin{split}
    \Div_y \uu^1 &= 0 \text{ in } \Omega \times Y,\\
    \DD \left( \uu^0 \right) + \DD_y \left( \uu^1 \right) &= 0 \text{   in } \Omega \times Y_s,\\
     \Div_y \emB^1 &= 0 \text{ in } \Omega \times Y,\\
    \Curl \emB^0 + \Curl_y \emB^1 &= 0 \text{ in } \Omega \times Y_f.
  \end{split}
\end{align}
These identities allow us to simplify \eqref{eq:50}-\eqref{eq:51} in
later calculations.

\subsection{The local problem}
\label{sec:local-problem}
The local problem is derived from \eqref{eq:50}-\eqref{eq:51} by
letting $\vv^0 = \emC^0 = \mathbf{0}$ and $q^0 = 0$, 
\begin{align}
\label{eq:53}
  \begin{split}
    &\frac{2}{\abs{Y}} \int_{\Omega} \int_Y \left( \DD(\uu^0(x)) +
      \DD_y(\uu^1(x,y)) \right): \DD_y (\vv^1(x,y))
    \di y \di x \\
    &\qquad {}+{} \frac{\nAl}{\nRm} \left\{ \frac{1}{\abs{Y}}
      \int_{\Omega} \int_Y \left( \Div \emB^0(x) + \Div_y \emB^1(x,y) \right)
      \cdot \Div_y \emC^1(x,y) \di y \di x \right.\\
    &\qquad + \left.  \frac{1}{\abs{Y}} \int_{\Omega} \int_{Y_s}
      \left( \Curl \emB^0(x) + \Curl_y \emB^1(x,y) \right) \cdot \Curl_y
      \emC^1(x,y) \di y \di x
    \right\}\\
    &{}+{} \frac{1}{\abs{Y}} \int_{\Omega} \int_Y p^0(x,y)
      \Div_y \vv^1(x,y) \di y \di x\\
    &\qquad -\nAl \frac{1}{\abs{Y}} \int_{\Omega} \int_{Y_s} \left(
      \uu^0(x)\times \emB^0(x) \right) \cdot 
      \Curl_y \emC^1(x,y) \di y \di x\\
    &= 0.
  \end{split}
\end{align}

Letting $\vv^1(x,y) = \ww(y) \varphi(x)$ and $\emC^1(x,y) = \emG (y)
\varphi (x)$, for $\ww, \emG \in H^1_{\per}(Y,\RR^d)$ and $\varphi \in \calD
(\Omega)$, we deduce from \eqref{eq:53} that, for a.e. $x \in \Omega$,
\begin{align}
\label{eq:54}
  \begin{split}
    & 2\int_Y \left( \DD(\uu^0(x)) +
      \DD_y(\uu^1(x,y)) \right): \DD_y (\ww(y))
    \di y  \\
    &\qquad {}+{} \frac{\nAl}{\nRm} \left\{ 
      \int_Y \left( \Div \emB^0(x) + \Div_y \emB^1(x,y) \right)
      \cdot \Div_y \emG (y) \di y \right.\\
    &\qquad + \left. \int_{Y_s}
      \left( \Curl \emB^0(x) + \Curl_y \emB^1(x,y) \right) \cdot \Curl_y
      \emG (y) \di y 
    \right\}\\
    &\qquad{}+{}  \int_Y p^0(x,y)
    \Div_y \ww (y) \di y  -\nAl \int_{Y_s} \left(
      \uu^0(x)\times \emB^0(x) \right) \cdot
    \Curl_y \emG(y) \di y \\
    &\qquad= 0.
  \end{split}
\end{align}
Define
\begin{align*}
  \calX_Y
  &\coloneqq
    \Set*{
    (\vec{\omega},\vec{\Theta}) \in
    H^1_{\per}(Y, \RR^d) \times H^1_{\per}(Y, \RR^d)
    \given
\begin{aligned}
  \Div_y \vec{\omega} &= 0 \text{ in }Y\\
  \DD_y(\vec{\omega}) &= 0 \text{ in } Y_s\\
  \Curl_y \vec{\Theta} &= 0 \text{ in } Y_f
\end{aligned}}
\end{align*}
So for $\left( \ww,\emG \right) \in \calX_Y$, from \eqref{eq:54} the
following holds a.e. $x \in \Omega$, 
\begin{align}
\label{eq:56}
\begin{split}
    & 2\int_Y \left( \DD(\uu^0(x)) +
      \DD_y(\uu^1(x,y)) \right): \DD_y (\ww (y))
    \di y  \\
    &\qquad {}+{} \frac{\nAl}{\nRm} \left\{ 
      \int_Y \left( \Div \emB^0(x) + \Div_y \emB^1 (x,y) \right)
      \cdot \Div_y \emG (y)\di y \right.\\
    &\qquad + \left.  \int_{Y_s}
      \left( \Curl \emB^0 (x) + \Curl_y \emB^1 (x,y)\right) \cdot \Curl_y
      \emG (y)\di y 
    \right\}\\
    &\qquad -\nAl \int_{Y_s} \left(
      \uu^0(x)\times \emB^0(x) \right) \cdot
    \Curl_y \emG(y) \di y \\
    &\qquad= 0.
  \end{split}
\end{align}
or equivalently,
\begin{align}
\label{eq:55}
  \begin{split}
    & 2\int_Y \DD_y(\uu^1): \DD_y (\ww) \di y {}+{} \frac{\nAl}{\nRm}
    \left\{ \int_Y \Div_y \emB^1 \cdot \Div_y \emG \di y + \int_{Y_s}
      \Curl_y \emB^1 \cdot \Curl_y \emG \di y
    \right\}\\
    &\qquad= -\int_Y  \DD(\uu^0) : \DD_y (\ww) \di y \\
    &\qquad \quad {}-{} \frac{\nAl}{\nRm} \left\{ \int_Y \Div \emB^0 \cdot
      \Div_y \emG \di y + \int_{Y_s} \Curl \emB^0 \cdot \Curl_y \emG
      \di y
    \right\}\\
    &\qquad \quad {}+{}\nAl \int_{Y_s} \left(
      \uu^0(x)\times \emB^0(x) \right) \cdot
    \Curl_y \emG(y) \di y.
  \end{split}
\end{align}

Clearly, for fix $x \in \Omega$, problem \eqref{eq:55} has a unique
solution
$\left( \uu^1 (x,\cdot), \emB^1(x,\cdot) \right) \in \calX_Y$, because
the left hand side of \eqref{eq:55} is coercive, which in turn, comes
from the inequality \eqref{eq:603} (note that this estimate also holds
for a convex polyhedron, which is why we can replace $\Omega$ by
$Y$). Therefore, as long as $\uu^0$ and $\emB^0$ are well-defined,
$\uu^1$ and $\emB^1$ are independent of the choice of subsequences
$\uu^{\varepsilon}$ and $\emB^{\varepsilon}$ in \eqref{eq:21}. Finally,
$p^0(x,\cdot) \in L^2_0(Y)$ is also unique due to the inf-sup condition
(repeating the first part of the proof of \cref{sec:stat-ferr-equat-8}).

First, we calculate $\uu^1$ in terms of $\uu^0$. In \eqref{eq:56}, let $\emG =
0$, then 
\begin{align}
\label{eq:58}
  \int_Y \left( \DD \left( \uu^0(x) \right) + \DD_y \left( \uu^1 (x,y)\right) \right) : \DD_y \left( \ww (y)\right) \di y
  &= 0
\end{align}
with $\ww \in H^1_{\per}\left( Y, \RR^d \right)$ satisfying $\Div_y
\ww = 0$ in $Y$ and $\DD_y(\ww) = 0$ in $Y_s$.
For $1\le i,j \le d$, define the function
$\UU^{ij} \coloneqq y_j\delta_{ik} \ee^k$; then, by direct
calculation, $\DD_y(\UU^{ij}) = \frac{1}{2} \left(
  \delta_{jm}\delta_{in} + \delta_{jn}\delta_{im} \right) \ee^n
\otimes \ee^m$. 
Let $\vec{\omega}^{ij}\in H^1_{\per}(Y,\RR^d)$ and $\pi^{ij}\in L^2_0(Y)$ be the
solutions of 
\begin{align}
\label{eq:59}
  \begin{split}
    \Div_y \left( \DD_y \left( \UU^{ij}-\vec{\omega}^{ij} \right) -
      \pi^{ij} \right)& = 0 \text{ in } Y_f,\\
    \Div_y \vec{\omega}^{ij}& = 0 \text{ in } Y_f,\\
    \DD_y \left(
      \UU^{ij}
      - \vec{\omega}^{ij} \right)& = 0 \text{ in } Y_s,\\
    \int_{\Gamma} \left( \DD_y \left( \UU^{ij}-\vec{\omega}^{ij}
      \right)- \pi^{ij}\II \right)\nn_{\Gamma} \di \hmeas &=0,\\
    \int_{\Gamma} \left( \DD_y \left( \UU^{ij}-\vec{\omega}^{ij}
      \right)- \pi^{ij}\II \right)\nn_{\Gamma}\times \nn_{\Gamma} \di
    \hmeas& =0.
  \end{split}
\end{align}

Then, integrating by parts \eqref{eq:58} and using \eqref{eq:57} and
\eqref{eq:59}, we see that $\uu^1$ is given by
\begin{align}
\label{eq:60}
  \uu^1(x,y) = -\left[ \DD \left( \uu^0(x) \right) \right]_{ij} \vec{\omega}^{ij}(y).
\end{align}

We now calculate $\emB^1$ in terms of $\emB^0$. In \eqref{eq:56}, let
$\ww = 0$ and use \eqref{eq:57} to obtain 
\begin{align}
\label{eq:61}
\int_{Y_s} \left( \Curl \emB^0(x) - \nRm \uu^0 (x)\times \emB^0(x) + \Curl_y
  \emB^1 (x,y)\right) \cdot \Curl_y \emG (y) \di y
  = 0,
\end{align}
with $\emG \in H^1_{\per}(Y,\RR^d)$ satisfying $\Curl_y \emG = 0$ in $Y_f$.
For $1 \le j \le d$, let $\vec{\Theta}^j \in H^1_{\per}(Y, \RR^d)$ and $\vec{\Psi}^j \in H^1_{\per}(Y, \RR^d)$ be the
solutions of 
\begin{align}
\label{eq:64}
  \begin{split}
    \Curl_y \Curl_y \left( \vec{\Theta}^j + \ee^j \right) &= 0 \text{
     in } Y_s,\\
    \left( \vec{\Theta}^j + \ee^j \right) \cdot \nn_{\Gamma} 
    &= 0
    \text{ on } \Gamma,\\
    \Div_y \vec{\Theta}^j
    &= 0 \text{ in }Y,
    \end{split}
    \begin{split}
    \Curl_y \left(  \vec{\Theta}^j + \ee^j\right) 
    &= 0 \text{ in }Y_f,\\
    \Curl_y \left( \vec{\Theta}^j + \ee^j \right)\times \nn_{\Gamma}
    &= 0 \text{ on }\Gamma,
    \end{split}
\end{align}
and 
\begin{align}
\label{eq:65}
  \begin{split}
    \Curl_y \Curl_y \left( \vec{\Psi}^j + \nRm \ee^j \right) 
    &= 0 \text{ in } Y_s,\\
     \left( \vec{\Psi}^j + \nRm\ee^j \right) \cdot \nn_{\Gamma} 
     &= 0
     \text{ on } \Gamma,\\
     \Div_y \vec{\Psi}^j
     &= 0 \text{ in }Y,
    \end{split}
    \begin{split}
    \Curl_y \vec{\Psi}^j 
    &= 0 \text{ in }Y_f,\\
    \Curl_y \left( \vec{\Psi}^j + \nRm\ee^j \right)\times \nn_{\Gamma}
    &= 0 \text{ on }\Gamma,
    \end{split}
\end{align}
respectively.  Then, integrating by parts \eqref{eq:61}, we see that
$\emB^1$ is given by
\begin{align}
\label{eq:62}
  \emB^1(x,y)
  = \epsilon_{ijk} \frac{\partial B_i^0}{\partial x_k}(x) \vec{\Theta}^j(y) -
  \epsilon_{ikj}u_i^0(x) B_k^0(x) \vec{\Psi}^j(y),
\end{align}
here $\epsilon_{ijk}$ is the (Levi-Civita) permutation symbol.

Now, we find a formula for $p^0 \in L^2_0(\Omega \times Y)$. Suppose 
\begin{align}
\label{eq:70}
  p^0(x,y)
  = 2 \left[ \DD \left( \uu^0(x) \right) \right]_{ij} \pi^{ij}(y) + \Pi(x,y)
\end{align}
for some $\Pi \in L^2_0(\Omega \times Y)$. We claim $\Pi$ is
independent of $y$. To see this, substitute \eqref{eq:60},
\eqref{eq:62} and use the local problems \eqref{eq:59}, \eqref{eq:64}, and \eqref{eq:65} in \eqref{eq:54} to obtain 
\begin{align*}
\int_Y p^0(x,y) \Div_y \ww (y) \di y = 0, \text{ for any } \ww \in H_{\per}^1(Y).
\end{align*}
Substituting \eqref{eq:70} into the above equation and integrating by
parts, all terms cancel by periodicity, except 
\begin{align*}
\int_Y \nabla_y \Pi^0(x,y) \cdot \ww(y) \di y = 0, \text{ for any } \ww
  \in H_{\per}^1(Y).
\end{align*}
Therefore, $\nabla_y\Pi(x,y) = \mathbf{0}$, i.e. $\Pi$ is independent of
$y$, and we write $\Pi(x,y) \equiv \Pi(x)$. Clearly, $p^{\varepsilon}
\wcv \Pi$ in $L^2(\Omega)$.

\subsection{The homogenized problem}
\label{sec:homogenized-problem}
The variational form of the homogenized equation is derived by letting
$\vv^1 = \emC^1 = \mathbf{0}$ and $q^1 = 0$ in
\eqref{eq:50}--\eqref{eq:51}, and then simplifying it by using \eqref{eq:57}:
\begin{align}
\label{eq:66}
\begin{split}
  &\frac{2}{\abs{Y}} \int_{\Omega} \int_Y \left( \DD(\uu^0(x)) +
    \DD_y(\uu^1(x,y)) \right): \DD(\vv^0(x)) \di y \di x \\
  &\qquad {}+{} \frac{\nAl}{\nRm} \frac{1}{\abs{Y}} \int_{\Omega}
  \int_{Y_s} \left( \Curl \emB^0(x) + \Curl_y \emB^1(x,y) \right) \cdot
  \Curl \emC^0 (x) \di y \di x\\
  &{}+{} \frac{1}{\abs{Y}} \int_{\Omega} \int_Y p^0(x,y) \Div \vv^0(x)
  \di y \di x\\
  &{}+{} \nRe \frac{1}{\abs{Y}} \int_{\Omega} \int_Y u^0_i(x) \left(
    \frac{\partial}{\partial x_i} \uu^0(x) + \frac{\partial}{\partial
      y_i} \uu^1(x,y) \right) \cdot \vv^0(x)\di y \di x\\
  &\qquad - \nAl \frac{1}{\abs{Y}} \int_{\Omega} \int_{Y_s}
    \Curl \emB^0(x) \times \emB^0(x)
  \cdot
  \vv^0(x) \di y \di x \\
  &\qquad -\nAl \frac{1}{\abs{Y}} \int_{\Omega} \int_{Y_s} \left(
    \uu^0(x)\times \emB^0(x) \right) \cdot \Curl \emC^0(x) \di y \di x\\
  &= \nRe \int_{\Omega} \bg (x) \cdot \vv^0(x) \di x + \nAl
  \frac{\abs{Y_s}}{\abs{Y}} \int_{\Omega} \hh(x) \cdot \emC^0(x) \di x,
  \end{split}
\end{align}
In \eqref{eq:66}, let $\emC^0 = 0$ and $\vv^0 \in
H_0^1(\Omega,\RR^d)$, we obtain 
\begin{align}
  \label{eq:67}
  \begin{split}
    &\frac{2}{\abs{Y}} \int_{\Omega} \int_Y \left( \DD(\uu^0(x)) +
      \DD_y(\uu^1(x,y)) \right): \DD(\vv^0(x)) \di y \di x \\
    &{}+{} \frac{1}{\abs{Y}} \int_{\Omega} \int_Y p^0(x,y) \Div \vv^0(x)
  \di y \di x\\
    &{}+{} \nRe \frac{1}{\abs{Y}} \int_{\Omega} \int_Y u^0_i (x)\left(
      \frac{\partial}{\partial x_i} \uu^0(x) +
      \frac{\partial}{\partial
        y_i} \uu^1(x,y) \right) \cdot \vv^0(x)\di y \di x\\
    &\qquad - \nAl \frac{\abs{Y_s}}{\abs{Y}} \int_{\Omega} \Curl
    \emB^0(x) \times \emB^0(x) \cdot
    \vv^0(x) \di x \\
    &= \nRe \int_{\Omega} \bg(x) \cdot \vv^0 (x)\di x.
  \end{split}
\end{align}
Define the \textit{effective viscosity} $\calN$, which is a fourth-rank tensor, by
\begin{align}
\label{eq:68}
  \calN_{ijmn}
  \coloneqq \frac{1}{\abs{Y}} \int_Y \left[\DD_y \left( \UU^{ij} -
  \vec{\omega}^{ij} \right) \right]_{mn} \di y.
\end{align}
Substituting \eqref{eq:60} and \eqref{eq:70} into \eqref{eq:67}, we obtain
\begin{align}
  \label{eq:69}
  \begin{split}
2\int_{\Omega} \calN_{ijmn} \left[ \DD \left( \uu^0 \right)
  \right]_{ij} \left[ \DD \left( \vv^0 \right) \right]_{mn} \di x +
  \int_{\Omega} \left( \uu^0 \cdot \nabla \right) \uu^0 \cdot \vv^0
  \di x + \int_{\Omega} \Pi \Div \vv^0 \di x\\
  - \nAl\frac{\abs{Y_s}}{\abs{Y}} \int_{\Omega} \Curl \emB^0 \times
  \emB^0 \cdot \vv^0 \di x = \nRe \int_{\Omega} \bg \cdot \vv^0 \di x.
  \end{split}
\end{align}
Here, we use the fact that 
$\int_Y \frac{\partial}{\partial y_i} \vec{\omega}^{ij} \di y= 0$ due
to periodicity, and that $\int_Y \pi^{ij} \di y = 0$ because $\pi^{ij} \in
L^2_0(Y)$. Integrating by parts \eqref{eq:69}, we have, on $\Omega$, that  
\begin{align}
\label{eq:71}
\left( \uu^0 \cdot \nabla \right)\uu^0 - \Div \left( 2  \calN_{ijmn} 
  \left[ \DD (\uu^0) \right]_{ij}\ee^m \otimes \ee^n - \Pi\, \II \right)
  =\nRe \bg + \nAl \frac{\abs{Y_s}}{\abs{Y}} \Curl \emB^0 \times \emB^0.
\end{align}

In \eqref{eq:66}, letting $\vv^0 = 0$ and $\emC^0 \in H_n^1(\Omega,\RR^d)$,
we obtain 
\begin{align}
\label{eq:72}
  \begin{split}
    &\frac{\nAl}{\nRm} \frac{1}{\abs{Y}} \int_{\Omega} \int_{Y_s}
    \left( \Curl \emB^0(x) + \Curl_y \emB^1(x,y) \right) \cdot
    \Curl \emC^0(x)  \di y \di x\\
    &\qquad -\nAl \frac{1}{\abs{Y}} \int_{\Omega} \int_{Y_s} \left(
      \uu^0(x)\times \emB^0(x) \right) \cdot \Curl \emC^0(x) \di y \di x\\
    &= \nAl \frac{\abs{Y_s}}{\abs{Y}} \int_{\Omega} \hh (x) \cdot \emC^0(x)
    \di x.
  \end{split}
\end{align}
Define the matrices $\calM$ and $\calE$, which represent the \emph{effective magnetic reluctivity} and the \emph{effective electric conductivity}, respectively, by
\begin{align}
  \label{eq:73}
  \begin{split}
  \calM_{jq}
  &\coloneqq \frac{1}{\abs{Y}} \int_{Y_s} \left[\left( \Curl_y\vec{\Theta}^j + \ee^j  \right) \right]_q \di y, \qquad
  \calE_{kq}
  \coloneqq \frac{1}{\abs{Y}} \int_{Y_s} \left[\left( \Curl_y\vec{\Psi}^k + \ee^k  \right) \right]_q \di y.
  \end{split}
\end{align}
Then, by substituting \eqref{eq:62} into \eqref{eq:72}, and using
\eqref{eq:73}, we obtain 
\begin{align}
\label{eq:74}
  \frac{1}{\nRm} \int_{\Omega} \calM_{jq} \epsilon_{ijk}\epsilon_{pqr}\frac{\partial
  B^0_i}{\partial x_k} \frac{\partial C^0_p}{\partial x_r} \di x-
  \int_{\Omega} \calE_{kq} \epsilon_{ijk} \epsilon_{pqr} u^0_i B^0_j \frac{\partial C_p^0}{\partial x_r} 
  \di x 
  = \frac{\abs{Y_s}}{\abs{Y}} \int_{\Omega} \hh \cdot \emC^0 \di x.
\end{align}
Using integration by parts, with $\emC^0 \in H_n^1(\Omega,\RR^d)$, we
conclude that, on $\Omega$,
\begin{align}
\label{eq:75}
  \frac{1}{\nRm} \Curl \left( \calM_{jn}\epsilon_{ijk} \frac{\partial B^0_i}{\partial x_k} \ee^n \right) - \Curl \left( \calE_{kn} \epsilon_{ijk} u^0_i B^0_j \ee^n \right) = \frac{\abs{Y_s}}{\abs{Y}} \hh.
\end{align}

In summary, from \eqref{eq:57}, \eqref{eq:71} and \eqref{eq:75}, we
obtain the macroscopic system that is about finding $\uu^0 \in H_0^1(\Omega,\RR^d)$, $\Pi \in L_0^2(\Omega)$, and $\emB^0 \in H_n^1(\Omega,\RR^d)$ satisfying on $\Omega$,
\begin{align} \label{eq:76}
 \begin{split}
 \Div \uu^0 = \Div \emB^0 &= 0,\\
\left( \uu^0 \cdot \nabla \right)\uu^0 - \Div \left( 2\calN_{ijmn} \left[ \DD (\uu^0) \right]_{ij}\ee^m \otimes \ee^n -      \Pi\, \II \right) 
&=\nRe \bg + \nAl \frac{\abs{Y_s}}{\abs{Y}}    \Curl \emB^0 \times \emB^0,\\
\frac{1}{\nRm} \Curl \left( \calM_{jn}\epsilon_{ijk} \frac{\partial B^0_i}{\partial x_k} \ee^n \right) - \Curl \left( \calE_{kn} \epsilon_{ijk} u^0_i B^0_j \ee^n \right) 
&= \frac{\abs{Y_s}}{\abs{Y}} \hh.
\end{split}
\end{align}

Here $\calN$, and $\calM$, $\calE$ are defined in \eqref{eq:68}, and
\eqref{eq:73}, respectively. It is worth to mention that by using the variational
formulation of the local problem \eqref{eq:59} and \eqref{eq:64}-\eqref{eq:65}, we
have 
\begin{align}
\label{eq:77}
  \begin{split}
    \calN_{ijmn} &= \frac{1}{\abs{Y}} \int_Y \DD_y \left(
      \UU^{ij}-\vec{\omega}^{ij} \right) : \DD_y \left(
      \UU^{mn}-\vec{\omega}^{mn} \right) \di y,\\
    \calM_{ij} &= \frac{1}{\abs{Y}} \int_{Y_s} \Curl_y \left( \vec{\Theta}^i +
      \ee^i \right) \cdot \Curl_y\left( \vec{\Theta}^j + \ee^j \right) \di y,\\
    \calE_{ij} &= \frac{1}{\abs{Y}} \int_{Y_s} \Curl_y\left( \vec{\Psi}^i +
      \ee^i \right) \cdot \Curl_y\left( \vec{\Psi}^j + \ee^j \right) \di y.
  \end{split}
\end{align}
Thus, the tensors are symmetric and elliptic. The well-posedness of
system \eqref{eq:76} now follows from the classical theory of
one-fluid magnetohydrodynamics,
c.f. \cite{giraultFiniteElementMethods2012,schotzauMixedFiniteElement2004,gunzburgerExistenceUniquenessFinite1991,gerbeauMathematicalMethodsMagnetohydrodynamics2006,guermondMixedFiniteElement2003}. By
the uniqueness of $\uu^0,\emB^0, \uu^1, \emB^1$ and $p^0$, we conclude
that the limits in \eqref{eq:21} hold for the full
sequence. \cref{sec:main-results} is proved.
\section{Conclusions}
\label{sec:conclusions}
The results obtained in Section~\ref{sec:main-results} demonstrate the \emph{effective} response of a viscous fluid with a locally periodic array of magnetic particles suspended in it. The original fine-scale problem is described by the system of equations \eqref{eq:5}-\eqref{eq:596}, and the effective equations are given by \eqref{eq:76}, in Section~\ref{sec:homogenized-problem}, with the effective coefficients defined by \eqref{eq:77}. As evident from the effective system obtained, these effective quantities depend on the instantaneous position of the particles, their geometry, and the magnetic and flow properties of the original suspension decoded in the cell problems \eqref{eq:59} and \eqref{eq:64}-\eqref{eq:65}. The effective medium is an \textit{incompressible electromagnetic fluid} described by the coupled set of Navier-Stokes and Maxwell's equations.  The effective Cauchy stress of the fluid is $\displaystyle  2\calN_{ijmn} \left[ \DD (\uu^0) \right]_{ij} \ee^m \otimes \ee^n -
      \Pi\, \II$, where $\calN$ is the effective viscosity, and the coupling between the homogenized fluid velocity $\uu$ and the homogenized magnetic field $\emB$ is given through the Lorentz force. The Maxwell's equations are represented by the combination of Amp\`ere's law, Ohm’s law, and Faraday's law, where the first two laws eliminate the electric field from the equation.
      
      It is worth mentioning that this paper is not concerned with \emph{modeling issues} for colloids with magnetizable particles, but rather focuses on the homogenization results. 
This study is the promised follow-up of the work in \cite{dangHomogenizationNondiluteSuspension2021} by the authors, where they considered a one-way coupling mechanism between the viscous fluid and the magnetic particles that are suspended in a viscous fluid and described by the linear relation between the magnetic flux density $\emB$ and the magnetic field strength $\emH$. In contrast to \cite{dangHomogenizationNondiluteSuspension2021}, this paper focuses on a \emph{non-linear} model of the given  magnetorheological fluid, where the two phases are interacting via the \emph{full (two-way) coupling} mechanism.
And, as in \cite{dangHomogenizationNondiluteSuspension2021}, the \emph{rigorous justification} of the obtained effective system is derived. This is also differing from previous contributions on the topic \cite{levySuspensionSolidParticles1983,nikaMultiscaleModelingMagnetorheological2020}, that dealt only with  formal asymptotics and did not consider the complicated non-linear model discussed in this paper.

\section*{Acknowledgements}
The work of the first
author was partially supported by NSF grant DMS-1350248.  The work of the third author was supported  by NSF grant DMS-2110036.  This
material is based upon work supported by and while serving at the
National Science Foundation for the second author Yuliya Gorb. Any
opinion, findings, and conclusions or recommendations expressed in
this material are those of the authors and do not necessarily reflect
views of the National Science Foundation.



\bibliographystyle{siamplain}
\bibliography{references}
\end{document}